\documentclass[reqno,12pt,letterpaper]{amsart}
\usepackage{amsmath,amssymb,amsthm,graphicx,mathrsfs,url}
\usepackage[usenames,dvipsnames]{color}
\usepackage[colorlinks=true,linkcolor=Red,citecolor=Green]{hyperref}

\def\?[#1]{\textbf{[#1]}\marginpar{\Large{\textbf{??}}}}

\newtheorem{theo}{Theorem}[section]
\newtheorem{prop}[theo]{Proposition}

\newtheorem{lemm}[theo]{Lemma}
\newtheorem{corr}[theo]{Corollary}
\newtheorem{rem}{Remark}
\numberwithin{equation}{section}

\newcommand{\mc}{\mathcal}
\newcommand{\rr}{\mathbb{R}}
\newcommand{\nn}{\mathbb{N}}
\newcommand{\cc}{\mathbb{C}}
\newcommand{\hh}{\mathbb{H}}
\newcommand{\zz}{\mathbb{Z}}

\newcommand{\la}{\lambda}
\newcommand{\eps}{\epsilon}

\newcommand{\pl}{\partial}
\newcommand{\x}{\times}

\newcommand{\til}{\widetilde}
\newcommand{\bbar}{\overline}

\newcommand{\cjd}{\rangle}
\newcommand{\cjg}{\langle}

\newcommand{\demi}{\tfrac{1}{2}}

\DeclareMathOperator{\Res}{Res}

\DeclareMathOperator{\supp}{supp}

\def\indic{\operatorname{1\hskip-2.75pt\relax l}}

\title[Classical and Quantum resonances]{Classical and quantum resonances for hyperbolic surfaces}
\author{Colin Guillarmou}
\email{cguillar@dma.ens.fr}
\address{DMA, U.M.R. 8553 CNRS, \'Ecole Normale Superieure, 45 rue d'Ulm,
75230 Paris cedex 05, France}

\author{Joachim Hilgert}
\email{hilgert@math.upb.de}
\address{Universit\"at Paderborn, Warburgerstr. 100, 33098 Paderborn, Germany}
\author{Tobias Weich}
\email{weich@math.upb.de}
\address{Universit\"at Paderborn, Warburgerstr. 100, 33098 Paderborn, Germany}
\begin{document}
\maketitle

\begin{abstract} For compact and for convex co-compact oriented hyperbolic surfaces, we prove 
an explicit correspondence between classical Ruelle resonant states and quantum resonant states, except at negative integers 
where the correspondence involves holomorphic sections of line bundles. 
\end{abstract}

\section{Introduction}
It is a classical result that on compact surfaces with constant negative curvature,
Selberg's trace formula allows to identify the eigenvalues of the Laplace 
operator $\Delta$ and certain zeros of the Selberg zeta function \cite{Se}, which can be entirely 
expressed in terms of closed geodesics. Later the same result has been 
established for convex-co-compact hyperbolic surfaces by 
Patterson-Perry \cite{PaPe} (see also \cite{BJP,GuZw3}), where the correspondence
is between certain zeros of the Selberg zeta function and the resonances of the Laplacian
(recall that convex co-compact hyperbolic surfaces are complete non-compact smooth Riemannian
surfaces of constant negative curvature with infinite volume). 
Both results show that on hyperbolic surfaces there is a 
deep connection between the spectral properties of the Laplacian (quantum mechanics)
and the properties of the geodesic flow (classical mechanics). However, the above 
results do not establish a link between the spectra of the Laplacian and a transfer
operator associated to the geodesic flow, nor do they establish a relation between
the associated resonant states. The aim of this article is to prove such an 
explicit correspondence. The 
previously known relation to the zeta zeros is a direct consequence of this 
correspondence.

Let us now introduce the concept of Ruelle resonances for the transfer operator associated to the geodesic flow on $M$. Let $M$ be a compact or convex co-compact 
hyperbolic surface and let $X$ be 
the vector field generating the geodesic flow $\varphi_t$ on the unit tangent 
bundle $SM$ of $M$. The linear operator
\[
 \mathcal L_t: \left\{ \begin{array}{ccc}
                       C_c^\infty(SM) &\to&C_c^\infty(SM)\\
                       f&\mapsto &f \circ \varphi_{-t}
                      \end{array}
\right.
\]
is called the \emph{transfer operator} of the geodesic flow and the vector field
$-X$ is its generator. The geodesic flow has the Anosov property, i.e. the tangent 
bundle of $SM$ splits into a direct sum  
\[ 
T(SM)=\rr X\oplus E_s\oplus E_u
\]
where $d\varphi_t$ is exponentially contracting in forward time (resp. backward time) 
on $E_s$ (resp. on $E_u$), and this decomposition is $\varphi_t$-invariant. The 
bundles $E_u$ and $E_s$ are smooth and there are two smooth non-vanishing
vector fields $U_\pm$ on $SM$ so that $E_s=\rr U_+$, $E_u=\rr U_-$, and 
$[X,U_\pm]=\pm U_\pm$. The fields $U_\pm$ generate the horocyclic flows.

For $f_1,f_2\in C_c^\infty(SM)$ we can define the correlation functions
\[ 
C_X(t;f_1,f_2):=\int_{SM}\mathcal L_t f_1. f_2d\mu_L
\]
where $\mu_L$ is the Liouville measure (invariant by $\varphi_t$). By 
\cite{BuLi,FaSj,DyZw} 
in the compact case and \cite{DyGu} in the convex co-compact case, the Laplace transform 
\[
R_X(\la;f_1,f_2):=-\int_0^\infty e^{-\la t}C_X(t; f_1,f_2)dt
\]
extends meromorphically from ${\rm Re}(\la)>0$ to $\cc$\footnote{The 
extension in a strip has been proved in \cite{Po, Ru}  before}. 
Notice that for 
${\rm Re}(\la)>0$ we have $R_X(\la;f_1,f_2)=\cjg(-X-\la)^{-1}f_1,f_2\cjd$ and 
$R_X(\la)$ gives a meromorphic extension of the Schwartz kernel of the resolvent of $-X$. 
The poles are called \emph{Ruelle resonances} and 
the residue operator $\Pi^X_{\la_0}:C_c^\infty(SM)\to \mc{D}'(SM)$
defined by 
\[ 
\cjg \Pi^X_{\la_0}f_1,f_2\cjd:={\rm Res}_{\la_0}R_X(\la;f_1,f_2), \quad \forall f_1,f_2\in C^\infty_c(SM)
\] 
has finite rank,  commutes with $X$, and $(-X-\la_0)$ is nilpotent on its range. 
The elements in the range of $\Pi_{\la_0}^X$ are called \emph{generalized Ruelle 
resonant states}. Note that by the results in \cite{BuLi, FaSj, DyZw, DyGu} the poles
can be identified with the discrete spectrum of $X$ in certain
Hilbert spaces and the generalized resonant states with generalized eigenfunctions.

The quantum resonances on $M$ can be introduced in a quite similar fashion, 
except that we have to work with the wave flow. 
Let $\Delta_M$ be the non-negative Laplacian and $U(t):=\cos(t\sqrt{\Delta_M-1/4})$ 
the wave operator on $M$. For $f_1,f_2\in C_c^\infty(M)$, we define the correlation function
\[ C_{\Delta}(t;f_1,f_2):=\int_M U(t)f_1.f_2\, {\rm dvol}.\]
Then, by standard spectral theory in the compact case, and by \cite{MaMe,GuZw} in 
the convex co-compact case, the Laplace transform
\[ 
R_\Delta(\la;f_1,f_2):=\frac{1}{1/2-\la}\int_{0}^\infty e^{(-\la +1/2)t}C_\Delta(t; f_1,f_2)dt
\]
extends meromorphically from ${\rm Re}(\la)>1$ to $\la\in \cc$. Notice that 
$R_{\Delta}(\la;f_1,f_2)=\cjg(\Delta_M-\la(1-\la))^{-1}f_1,f_2\cjd$ for ${\rm Re}(\la)>1$ and $R_\Delta(\la)$ is 
a meromorphic extension of the resolvent of $\Delta_M$.
The poles are called \emph{quantum resonances} and 
the residue operator $\Pi^\Delta_{\la_0}:C_c^\infty(M)\to C^\infty(M)$ defined by 
\[
\cjg \Pi^\Delta_{\la_0}f_1,f_2\cjd:={\rm Res}_{\la_0}R_\Delta(\la;f_1,f_2), \quad \forall f_1,f_2\in C^\infty_c(M)
\] 
has finite rank, commutes with $\Delta_M$, and $(\Delta_M-\la_0(1-\la_0))$ is nilpotent 
on its range. The elements in the range of $\Pi_{\la_0}^\Delta$ are called 
\emph{generalized quantum resonant states}. 

Note that the Ruelle generalized resonant states are distributions on $SM$ while the 
quantum resonant states are functions on $M$. In order to formulate the explicit 
correspondence between them, we consider the projection 
$\pi_0:SM\to M$ on the base and 
${\pi_0}_*:\mc{D}'(SM)\to \mc{D}'(M)$ the operator dual to the pull-back 
$\pi_0^*:C_c^\infty(M)\to C_c^\infty(SM)$ (note that ${\pi_0}_*$ corresponds to integration in the fibers of $SM$). In order to state the theorems, we also need 
to introduce the canonical line bundle $\mc{K}:=(T^*M)^{1,0}$
and its dual $\mc{K}^{-1}:=(T^*M)^{0,1}$, and we denote their  tensor powers 
by $\mc{K}^n:=\mc{K}^{\otimes n}$ and 
$\mc{K}^{-n}:=(\mc{K}^{-1})^{\otimes n}$. Here the Riemann surface is oriented and thus inherits a complex structure.  
 Then there is a natural map
\[
\pi^*_n:C_c^\infty(M;\mc{K}^{n})\to C_c^\infty(SM), \quad \pi_n^*f(x,v):=f(x)(\otimes^n v)
\]
and we consider its dual operator ${\pi_n}_*:\mc{D}'(SM)\to \mc{D}'(M;\mc{K}^n)$ which can be viewed
 as the $n$-th Fourier component in the fibers of $SM$ ($SM$ is equipped with the Liouville measure and $M$ with the Riemannian volume form to define the pairings).

Let us formulate the first main result: 
\begin{theo} \label{compactcase}
Let $M=\Gamma\backslash \hh^{2}$ be a smooth oriented compact hyperbolic surface 
and let $SM$ be its unit tangent bundle. Then\\
1) for each $\la_0\in \cc\setminus (-\demi\cup  -\nn)$  the pushforward map 
${\pi_0}_*$ restricts to a linear isomorphism of complex vector spaces
\[
{\pi_0}_* : {\rm Ran}(\Pi_{\la_0}^X)\cap \ker U_- \to {\rm Ran}\, \Pi_{\la_0+1}^\Delta =\ker (\Delta_M+\la_0(1+\la_0)) 
\]
2) for $\la_0=-\demi$, the map 
\[
{\pi_0}_* : {\rm Ran}(\Pi_{\la_0}^X)\cap \ker U_- \to \ker (\Delta_M-\tfrac{1}{4})
\]
is surjective and has a kernel of complex dimension  $\dim \ker (\Delta_M-\tfrac{1}{4})$. In other words, the multiplicity of $-1/2$ as a Ruelle resonances is equal to twice the multiplicity of $1/4$ as eigenvalue of $\Delta_M$.\\
3) For $\la_0=-n\in -\nn$, the following map is an isomorphism of complex vector spaces
\[
 {\pi_n}_*\oplus {\pi_{-n}}_* : {\rm Ran}(\Pi_{-n}^X)\cap \ker U_-\to H_{n}(M)\oplus H_{-n}(M)
\]
with $H_n(M):=\{u\in C^\infty(M; \mc{K}^n); \bbar{\pl}u=0\}$,  
$H_{-n}(M):=\{u\in C^\infty(M; \mc{K}^{-n}); \pl u=0\}$. Consequently, the complex 
dimension of ${\rm Ran}(\Pi_{-n}^X)\cap \ker U_-$ is $(2n-1)|\chi(M)|$ if $n>1$ and $|\chi(M)|+2$ if $n=1$, where $\chi(M)$ is the Euler characteristic.
\end{theo}

Theorem \ref{compactcase} gives a full characterization of the Ruelle resonant 
states, that are invariant under the horocyclic flow. We call these resonances 
\emph{the first band} of Ruelle resonances.
Part 1) of Theorem \ref{compactcase}
has been proved in \cite{DFG}  
in any dimension. We slightly simplify the argument and 
characterize all first band resonant states including 
the particular points  $\la_0\in -1/2-\nn_0/2$ which were left out in \cite{DFG}. 
Theorem~\ref{compactcase} is also related to the 
classification of horocyclic invariant distributions by \cite{FlFo} but we use a 
different approach avoiding the Plancherel decomposition into unitary irreducible 
representations. This has the advantage that our approach gives a more geometric 
description of the resonant states and that it extends to the convex co-compact setting.
In Theorem~\ref{noninteger}  we give slightly more precise
statements including the description of possible Jordan blocks in the Ruelle spectrum. 

To state our result in the convex co-compact setting, let us first recall some geometric definitions.
A convex co-compact hyperbolic surface $M$ can be realized as a quotient 
$M=\Gamma\backslash \hh^2$ of the hyperbolic plane where 
$\Gamma\subset {\rm PSL}_2(\rr)$ a discrete subgroup whose non-trivial elements are  
hyperbolic transformations. Viewing $\Gamma$ as a subgroup of 
${\rm PSL}_2(\cc)$, it also acts by conformal transformations on the 
Riemann sphere $\bbar{\cc}\simeq \mathbb{S}^2$, and the action is free 
and properly discontinuous on the complement of the limit set 
$\Lambda_\Gamma\subset \rr$, which can be defined as the closure of 
the set of fixed points of non-trivial elements $\gamma\in\Gamma$. The quotient 
$M_2:=\Gamma\backslash (\bbar{\cc}\setminus \Lambda_\Gamma)$ is a compact 
Riemann surface containing two copies $M_\pm $ of $M$ corresponding to 
$M_\pm:=\Gamma\backslash \{\pm {\rm Im}(z)>0\}$, 
and $\bbar{M}:=\Gamma\backslash (\{{\rm Im}(z)\geq 0\}\setminus \Lambda_\Gamma)$ 
provides a smooth conformal compactification to $M$ in which $\pl\bbar{M}$ represents 
the geometric infinity of $M$. The surface $M_2$ has an involution $\mc{I}$ 
induced by $z\mapsto \bar{z}$ and fixing $\pl \bbar{M}$. 

\begin{theo} \label{convexcocompactcase}
Let $M=\Gamma\backslash \hh^{2}$ be a smooth oriented convex co-compact hyperbolic 
surface 
and let $SM$ be its unit tangent bundle. Then\\
1) for each $\la_0\in \cc\setminus -\nn$ the pushforward map 
${\pi_0}_*$ restricts to a linear isomorphism of complex vector spaces
\[
{\pi_0}_* : {\rm Ran}(\Pi_{\la_0}^X)\cap \ker U_- \to {\rm Ran}(\Pi_{\la_0+1}^\Delta).
\]
If $\la_0\in -1/2-\nn$, ${\rm Ran}(\Pi_{\la_0}^X)\cap \ker U_-=0$ and 
${\rm Ran}(\Pi_{\la_0+1}^\Delta)=0$.
\\
2) For $\la_0=-n\in -\nn$, the following map is an isomorphism of real vector spaces if $\Gamma$ is not cyclic:
\[
 {\pi_n}_*: {\rm Ran}(\Pi_{-n}^X)\cap \ker U_-\to H_{n}(M).
\]
Here $H_n(M):=\{f|_{\bbar{M}}; f\in C^\infty(M_2; \mc{K}^n), \bbar{\pl}f=0, \mc{I}^*f=\bar{f}\}$ and $\mc{I}:M_2\to M_2$ is the natural involution fixing $\pl\bbar{M}$ in $M_2$. Consequently the real dimension of ${\rm Ran}(\Pi_{-n}^X)\cap \ker U_-$ is $(2n-1)|\chi(M)|$ if $n>1$ and $|\chi(M)|+1$ if $n=1$.
\end{theo}
As in the compact case we provide more precise information on the correspondence between Jordan blocks, see Theorem \ref{classicquantic}. In 2), we note that ${\pi_0}_*({\rm Ran}(\Pi_{-n}^X))=0$ for non-cyclic groups, ${\rm Ran}(\Pi_{-n+1}^\Delta)=0$ for $n>1$ and 
${\rm Ran}(\Pi_{0}^\Delta)$ is the space of constant functions (see the comments after Corollary \ref{selberg}).
In the case of a cyclic group, the classical resonances at $-n$ have dimension $2$ and no Jordan blocks, and we get that 
${\pi_0}_*: {\rm Ran}(\Pi_{-n}^X)\mapsto {\rm Ran}(\Pi_{-n+1}^\Delta)$ 
has a kernel of dimension $1$ and image of dimension $1$, while ${\rm Ran}(\Pi_{-n+1}^\Delta)$ has dimension $2$ and there is a Jordan block for the quantum resonance. The map ${\pi_n}_*: {\rm Ran}(\Pi_{-n}^X)\cap \ker U_-\cap \ker{\pi_0}_*\to H_{n}(M)$ is one-to-one in that case.

Beyond the description of the first band of Ruelle resonances we obtain a description
of the full spectrum of Ruelle resonances by applying the vector fields $U_+$ iteratively.
\begin{prop}\label{thm:higher_bands}
 Let $M$ be a compact or convex co-compact hyperbolic surface, then
 \[
  {\rm Ran}(\Pi_{\la_0}^X) = \bigoplus_{0\leq \ell\leq  |{\rm Re}(\lambda_0)|} U_+^\ell \Big({\rm Ran}(\Pi_{\la_0+\ell}^X) \cap \ker U_-\Big)
 \]
 and the map $U_+^\ell :{\rm Ran}(\Pi_{\la_0+\ell}^X) \cap \ker U_- \to {\rm Ran}(\Pi_{\la_0}^X)$
 is injective unless $\lambda_0 +\ell =0$.
\end{prop}
The case $\lambda_0 +\ell =0$ can 
 only occur if $M$ is compact, in which case 
 ${\rm Ran}(\Pi_{0}^X)$ is the space of constant functions, killed by the differential 
 operator $U_+$.

As a consequence, in Section \ref{sec:zeta_functions}, we obtain an alternative proof  
of the results \cite{PaPe,BJP} on the zeros of the Selberg zeta function in our situation.

We end the introduction by a rough outline of the proofs of Theorem \ref{compactcase}
and \ref{convexcocompactcase}: a central ingredient is the microlocal 
characterization of
Ruelle resonant states in \cite{FaSj,DyGu}. In these references it has
been shown that a distribution $u\in \mathcal D'(SM)$ is a generalized
Ruelle resonant state for a Ruelle resonance $\la_0 \in \cc$ (i.e. $u\in {\rm Ran}\, \Pi^X_{\la_0}$)
if and only if there exists $j\geq 1$ such that $(-X-\la_0)^ju=0$ and
\[
\begin{split}
{\rm WF}(u)\subset E_u^* &\quad  \textrm{ if } M \textrm{ is compact}\\
{\rm WF}(u)\subset E_u^*,\,\, {\rm supp}(u)\subset \Lambda_+ &  
\quad  \textrm{ if } M \textrm{ is convex co-compact},
\end{split}
\]  
Here ${\rm WF}(u)\subset T^*(SM)$ denotes the wave-front set of $u$ and 
$E_u^*\subset T^*(SM)$ is the subbundle defined by $E_u^*(\rr X\oplus E_u)=0$.
Furthermore, on convex co-compact surfaces we use the notation
\[
\Lambda_\pm:=\{ y\in SM;  d(\pi_0(\varphi_{t}(y)),x_0)\not\to \infty \textrm{ as }t\to \mp\infty \}
\]
where $d$ denotes the Riemannian distance and $x_0\in M$ is any fixed point. Note that the set $\Lambda_+$ (resp. $\Lambda_-$) has a clear dynamical 
interpretation as it corresponds to trajectories that do not escape to infinity 
in the past (resp. in the future).

Using this characterization we follow the general strategy of \cite{DFG}. We
consider the hyperbolic surface as a quotient $M=\Gamma\backslash\hh^2$ 
of its universal cover $\hh^2$ by a co-compact, respectively convex co-compact, 
discrete subgroup $\Gamma\subset {\rm PSL}_2(\rr)$. If we lift the horocyclic invariant 
Ruelle resonant states to $\hh^2$, we can relate them to distributions on 
$\mathbb{S}^1=\pl\bbar{\hh}^2$, conformally covariant by the 
group $\Gamma$ and supported in $\Lambda_\Gamma$.  
We then show that such distributions are in correspondence with quantum resonant 
states using the bijectivity of the Poisson-Helgason transform at the 
noninteger points.
While this step is straightforward
for compact surfaces, the convex co-compact setting is more complicated. A central ingredient is a characterization 
of generalized quantum resonant states using their asymptotic behavior towards the boundary.
One can show (see Proposition \ref{caractreson}) that $u\in {\rm Ran}\, \Pi^\Delta_{\la_0}$ 
if and only if there exists $j\geq 1$ such that $(\Delta-\la_0(1-\la_0))^ju=0$ and
\begin{equation}\label{eq:quantum_res_cond}
\begin{split}
u\in C^\infty(M) &\quad  \textrm{ if } M \textrm{ is compact},\\
u\in \bigoplus_{k=0}^{j-1}\rho^{\la_0}\log(\rho)^kC_{\rm ev}^\infty(\bbar{M}) &  
\quad  \textrm{ if } M \textrm{ is convex co-compact},
\end{split}
\end{equation}
where $C_{\rm ev}^\infty(\bbar{M})$ denotes the space of smooth functions on 
$\bbar{M}$ which extend smoothly to $M_2$ as even functions with respect to the 
involution $\mc{I}$, and
$\rho\in C^\infty_{\rm ev}(\bbar{M})$ is a boundary defining function of 
$\pl\bbar{M}$ in $\bbar{M}$. We prove that the asymptotic condition \eqref{eq:quantum_res_cond}
corresponds to the fact that the associated distribution $\omega\in \mathcal D'(\mathbb{S}^1)$ via the Poisson-Helgason transform is supported in the limit set $\Lambda_\Gamma\subset \mathbb{S}^1$. Analogously, the condition that a horocyclic invariant
Ruelle resonant state is supported in $\Lambda_+$ is equivalent
to the fact that its associated distribution $\omega\in \mathcal D'(\mathbb{S}^1)$ is 
again supported on the limit set.

For the negative integer points, the Poisson-Helgason transform fails to be 
bijective, thus there might be equivariant boundary distributions
$\omega\in\mathcal D'(\mathbb S^1)$ and consequently Ruelle resonances 
that are not related to Laplace eigenfunctions and lie in the 
kernel of the Poisson-Helgason transform. In order to characterize these 
distributions we consider a vector valued Poisson-Helgason transform, 
whose image is contained in the sections of the complex line bundle $\mathcal K^n$. Roughly 
speaking we prove that a combination of these vector valued Poisson
transformation yields an equivariant bijection between the kernel
of the Poisson-Helgason transform at the negative integer points and the 
holomorphic respectively antiholomorphic sections in $\mathcal K^n$ and
$\mathcal K^{-n}$. Again the convex co-compact part is more complicated, 
as we also have to take care of the asymptotics at 
the boundary.

Note that the invariant distributions supported on the limit set which appear
as an intermediate step in our proof were also studied in 
\cite{BuOl,BuOl2}, and our result somehow completes the picture. 

\textbf{Acknowledgements.} C.G. is partially supported by ANR-13-BS01-0007-01 and ANR-13-JS01-0006. 
J.H. and T.W. acknowledge financial support by the DFG grant DFG HI-412-12/1. 
We thank S. Dyatlov, F.Faure and M. Zworski for useful discussions, and 
Viet for suggesting the title.
 
\section{Geodesic flow on hyperbolic manifolds}
In this Section, we recall a few facts about the geodesic flow, horocyclic 
derivatives and the Poisson operator on the real hyperbolic plane that are needed for the next sections. 
We refer to the paper \cite{DFG} where all the material is described in full detail. 

\subsection{Hyperbolic space}\label{hyperbolicspace}
Let $\hh^{2}$ be the real hyperbolic space of dimension $2$, which we view as the open 
unit ball in $\rr^{2}$ equipped with the metric $g_{\hh^2}:=\frac{4|dx|^2}{(1-|x|^2)^2}$. 
The unit tangent bundle
is denoted by $S\hh^{2}$  and the projection is denoted by $\pi_0: S\hh^{2}\to \hh^{2}$ on the base. 
The hyperbolic space $\hh^{2}$ is compactified smoothly into the closed unit ball of 
$\rr^{2}$, denoted by $\bbar{\hh}^{2}$ and its boundary is the unit sphere $\mathbb{S}^1$. 
 
Let $X$ be the geodesic vector field on $S\hh^{2}$ and $\varphi_t:S\hh^{2}\to S\hh^{2}$ be 
the geodesic flow at time $t\in\rr$.  We denote 
by $B_\pm : S\hh^{2}\to \mathbb{S}^1$ the endpoint maps assigning to a vector 
$(x,v)\in S\hh^{2}$ the endpoint on $\mathbb{S}^1$ of the geodesic passing through 
$(x,v)$ in positive time (+) and negative time (-). These maps are submersions and  
allow 
to identify $S\hh^{2}$ with $\hh^{2}\x S^1$ by the map $(x,v)\mapsto (x,B_\pm (x,v))$.
It is easy to compute $B_\pm$ explicitly: using the complex coordinate $z=x_1+ix_2\in \cc$ 
for the point $x=(x_1,x_2)$ (with $|x|<1$) and identifying $v\in S_z\hh^2$ with $e^{i\theta}$ through 
$2v/(1-|z|^2)=\cos(\theta)\pl_{x_1}+\sin(\theta)\pl_{x_2}$, 
we get 
\begin{equation}\label{formulaBpm}
B_-(z,e^{i\theta})=\frac{-e^{i\theta}+z}{-e^{i\theta}\bar{z}+1}, \quad 
B_+(z,e^{i\theta})=\frac{e^{i\theta}+z}{e^{i\theta}\bar{z}+1}.
\end{equation}
For each $z$, the map $B_z: e^{i\theta}\mapsto B_-(z,e^{i\theta})$ is a 
diffeomorphism of $\mathbb{S}^1$
and its inverse is given by 
\begin{equation}\label{formulaBinv}
B_z^{-1}(e^{i\alpha})=-e^{i\alpha}\frac{ze^{-i\alpha}-1}{\bar{z}e^{i\alpha}-1}.
\end{equation}
There exists two positive functions $\Phi_\pm\in C^\infty(S\hh^{2})$ 
satisfying $X\Phi_\pm=\pm \Phi_{\pm}$, given by 
\begin{equation}\label{phipm} 
\Phi_\pm(x,v):= P(x,B_\pm (x,v))  
\end{equation}  
where $P(x,\nu)$ is the Poisson kernel given by 
\begin{equation}\label{poissonker} 
P(x,\nu):=\frac{1-|x|^2}{|x-\nu|^2}, \quad x\in \hh^{2}, \nu \in \mathbb{S}^1.
\end{equation}

The  group of orientation preserving isometries of $\hh^{2}$ is the group
\[
G:={\rm PSU}(1,1)\simeq {\rm PSL}_2(\rr)
\]
An element $\gamma\in G\subset {\rm PSL}_2(\cc)$ acts on $\cc$ by M\"obius 
transformations and preserves the unit ball $\hh^{2}$, and this action preserves 
also the closure $\bbar{\hh}^{2}$. Furthermore the $G$ action on $\hh^2$ lifts linearly
to an action on $T\hh^2$ and as the action on the base space $\hh^2$ is isometric
it can be restricted to $S\hh^2$.
By abuse of notation, for $\gamma\in G$, we also denote 
the action of $\gamma$ on  $S\hh^2$ or  $\mathbb{S}^1$ by the same letter $\gamma$. 
By $|d\gamma|:\mathbb{S}^1\to\rr$ we denote the norm of 
the differential $d\gamma$ on the boundary $\mathbb{S}^1=\pl \bbar{\hh}^{2}$ of 
the unit ball with respect to the Euclidean norm.
Note that the above defined functions $\Phi_\pm$ and maps $B_\pm$ are compatible 
with respect to these $G$ actions in the sense that one has the relations
\begin{equation}\label{changeofPhi}
\gamma^*\Phi_\pm(x,v)=\Phi_{\pm}(x,v)N_\gamma(B_\pm(x,v)), \quad B_\pm(\gamma(x,v))=\gamma(B_\pm(x,v))\end{equation}  
where $N_\gamma(\nu):=|d\gamma(\nu)|^{-1}$.

As the $G$ action on $S\hh^2$ is free and transitive, we can identify $G\simeq S\hh^2$
via the natural isomorphism 
\begin{equation}\label{GSH^2}
 G\to S\hh^2 , \quad \gamma \mapsto (\gamma(0),\demi d\gamma(0).\pl_x).
\end{equation}
The Lie algebra $\mathfrak g =\mathfrak{sl}_2(\rr)$ of $G$ is spanned by 
\begin{equation}\label{basis} 
U_+:=\begin{pmatrix} 0&1\\0&0\end{pmatrix},\quad
U_-:=\begin{pmatrix} 0&0\\1&0\end{pmatrix}, \quad X:=\begin{pmatrix} {1\over 2}&0\\0&-{1\over 2}\end{pmatrix}.\end{equation}
These elements can also be viewed as left invariant smooth vector fields on 
$G\simeq S\hh^2$, which 
form at any point $(x,v)\in S\hh^2$ a basis of $T(S\hh^2)$, and the following 
commutation relations hold 
\begin{equation}\label{commutation}
[X,U_\pm]= \pm U_\pm , \quad [U_+,U_-]=2X.
\end{equation}
The geodesic vector field is represented by $X$ and we call $U_+$ the stable 
derivative and $U_-$ the unstable derivative. The vector fields  $X,U_\pm$ can 
be viewed as first order linear differential operators on $S\hh^2$, thus acting 
on distributions, and by \eqref{commutation}, $X$ preserves $\ker U_\pm$. 
Another decomposition that is quite natural for $T(S\hh^2)$ is 
\[ X:=\begin{pmatrix} {1\over 2}&0\\0&-{1\over 2}\end{pmatrix}, \quad
X_\perp:=\begin{pmatrix} 0&{1\over 2}\\{1\over 2}&0\end{pmatrix}, \quad 
V:=\begin{pmatrix} 0& {1\over 2}\\ -{1\over 2}& 0\end{pmatrix}.\]
which satisfy $U_\pm=X\pm X_\perp$ and 
\begin{equation}\label{commutation2}
[X,V]=X_\perp , \quad [X,X_\perp]=V, \quad [V,X_\perp]= X.
\end{equation}
The vector field $V$ generates the ${\rm SO}(2)$ action on $G$ and 
geometrically, it generates the 
rotation in the fibers of $S\hh^2$ (that are circles); it is called the vertical 
vector field since $d\pi_0(V)=0$.
For what follows, we will always view $X,V,X_\perp, U_+,U_-$ as vector fields on $S\hh^2$.

There is a smooth splitting of $T(S\hh^{2})$ into flow, stable and unstable bundles, 
\[
T(S\hh^2)= \rr X\oplus E_s\oplus E_u\]
with the property that there is $C>0$ uniform such that 
\begin{equation}\label{stableunstable} 
||d\varphi_t(y).w||\leq Ce^{-|t|}||w||, \;\; \forall t\geq 0 , \forall w\in E_s(y) \textrm{ or }\forall t\leq 0, \forall w\in E_u(y).
\end{equation}
Here the norm on $S\hh^{2}$ is with respect to the Sasaki metric.
The space $E_s$ is generated by the vector field $U_+$ and $E_u$ by the vector field $U_-$ where 
$U_\pm$ are the images by the map \eqref{GSH^2}  of the left invariant vector fields  
in \eqref{basis}.

There are two  important properties of $\Phi_\pm$ with respect to stable/unstable derivatives:
\begin{equation}\label{Upmkills}
U_\pm \Phi_{\pm}=0 , \textrm{ and }
dB_{\pm}.U_\pm=0. \end{equation}
Let $\mc{Q}_\pm: \mc{D}'(\mathbb{S}^1)\to \mc{D}'(S\hh^{2})$ be the pull-back by $B_\pm$ 
acting on distributions which is well defined since $B_\pm$ are submersions.
It is a linear isomorphism between the following spaces (see \cite[Lemma 4.7]{DFG})
\begin{equation}\label{Qpmiso}
\mc{Q}_\pm : \mc{D}'(\mathbb{S}^1) \to \mc{D}'(S\hh^{2})\cap \ker U_\pm\cap \ker X.
\end{equation}

\subsection{Poisson-Helgason transform}
We say that a smooth function 
$f$ on $\hh^{2}$ is \emph{tempered} if there exists $C>0$ such that  
$|f(x)|\leq Ce^{Cd_{\hh^2}(x,0)}$ if 
$0$ is the center of $\hh^{2}$ (viewed as the unit disk) and $d_{\hh^2}(\cdot,\cdot)$ 
denotes the hyperbolic distance. 
Below, we view the space of distributions $\mc{D}'(\mathbb{S}^1)$ on $\mathbb{S}^1$ as the topological dual 
of $C^\infty(\mathbb{S}^1)$ and we embed $C^\infty(\mathbb{S}^1)\subset \mc{D}'(\mathbb{S}^1)$ by the pairing
\[ \cjg \omega, \chi \cjd_{\mathbb{S}^1}:=\int_{\mathbb{S}^1} \omega(\nu)\chi(\nu)dS(\nu) \]
where the measure $dS$ is the standard measure on $\mathbb{S}^1$ (viewed as a the unit circle in $\rr^2$).  
Then the following result was proved\footnote{Note that this was proved in the setting of hyperfunctions by Helgason \cite{He}} in \cite[Corollary 11.3 and Theorem 12.2]{VdBSc} and \cite[Theorem 3.15]{OsSe} but we follow the presentation given in 
\cite[Section 6.3]{DFG}:
\begin{lemm}\label{poissoniso} 
For $\la\in \cc$, let $\mc{P}_\la: \mc{D}'(\mathbb{S}^1)\to C^\infty(\hh^{2})$ 
be the Poisson-Helgason transform
\[\mc{P}_\la( \omega)(x):={\pi_0}_* (\Phi_-^{\la}\mc{Q}_-(\omega))(x)=\cjg \omega, P^{1+\la}(x,\cdot)\cjd_{\mathbb{S}^1}
\]
where $P(x,\nu)$ is the Poisson kernel of \eqref{poissonker} and ${\pi_0}_*$ is the adjoint of the pull-back 
$\pi_0^*: C_c^\infty(\hh^{2})\to C_c^{\infty}(S\hh^{2})$.  Then $\mc{P}_\la$ maps 
$\mc{D}'(\mathbb{S}^1)$ onto the space of tempered functions in the kernel of 
$(\Delta_{\hh^2}+\la(1+\la))$, where $\Delta_{\hh^2}=d^*d$ is the positive Laplacian acting on functions on $\hh^{2}$ and if $\la\notin-\nn$, $\mc{P}_\la$ is an isomorphism.
Finally, if $\gamma\in G$ is an isometry of $\hh^{2}$, we have the relation
$\gamma^*\mc{P}_\la(\omega)=\mc{P}_\la(|d\gamma|^{-\la}\gamma^*\omega)$ for each $\omega\in \mc{D}'(\mathbb{S}^1)$.
\end{lemm}
It is useful to describe the inverse of $\mc{P}_{\la}$ when $\la\notin -\nn$. 
For this purpose we can use for instance \cite[Lemma 6.8]{DFG}. 
First, if  $\la\notin -\nn$ and $\omega\in \mc{D}'(\mathbb{S}^1)$, for each
$\chi\in C^\infty(\mathbb{S}^1)$ and $t\in (0,1)$ one has\footnote{The case $\la=-1/2$ is not really studied in 
\cite[Lemma 6.8]{DFG} but the analysis done there for $\la\in -1/2+\nn$ applies as well for $\la=-1/2$
by using the explicit expression of the modified Bessel function $K_{0}(z)$ as a converging series.} 
\begin{equation}\label{weakasymptotic0} 
\int_{\mathbb{S}^1} \mc{P}_{\la}(\omega)(\tfrac{2-t}{2+t}\nu)\chi(\nu)dS(\nu)=
\left \{\begin{array}{ll}
t^{-\la}F_{\la}^-(t)+t^{\la+1}F_{\la}^+(t) & \textrm{if } \la \notin -1/2+\zz\\ 
t^{-\la}F_{\la}^-(t)+t^{\la+1}\log(t)F_{\la}^+(t) & \textrm{if } \la\in -1/2+\nn_0\\
t^{-\la}\log(t)F_{\la}^-(t)+t^{\la+1}F_{\la}^+(t) & \textrm{if } \la\in -1/2-\nn
\end{array}\right.  
\end{equation}
where $F_\la^\pm\in C_{\rm ev}^\infty([0,1))$, and $C_{\rm ev}^\infty([0,1))$ is the subset of 
$C^\infty([0,1))$ consisting of functions with an even Taylor expansion at $0$\footnote{The evenness of the expansion at $t=0$ comes directly from the proof in \cite[Lemma 6.8]{DFG} when acting on functions, since 
the special functions appearing in the argument are Bessel functions that have even expansions.}. 
The exact expressions of $F_\la^\pm (0)$ can be obtained directly  from the study of the Poisson operator in \cite{GrZw} and the computation of the scattering operator $\mc{S}(s)$ of $\hh^{2}$ in \cite[Appendix]{GuZw}. 
The scattering operator is defined as the operator acting on $C^\infty(\mathbb{S}^1)$ given by the explicit function of the Laplacian on $\mathbb{S}^1$ 
\begin{equation}\label{formulaS(s)} 
\mc{S}(s):=\frac{\Gamma\Big(\sqrt{\Delta_{\mathbb{S}^1}}+s\Big)}{\Gamma\Big(\sqrt{\Delta_{\mathbb{S}^1}}+1-s\Big)},
\end{equation}
with Schwartz kernel on $\mathbb{S}^1$ given for ${\rm Re}(s)<1/2$ by 
\[ \mc{S}(s; \nu,\nu')=\pi^{-\demi}\frac{2^s\Gamma(s)}{\Gamma(-s+\demi)}|\nu-\nu'|^{-2s}.\]
It is a holomorphic family of operators in $s\notin -\nn_0$ with poles of order $1$ at $-\nn_0$, which is an isomorphism on $\mc{D}'(\mathbb{S}^1)$ outside the poles and satisfies the following functional equations 
\begin{equation}\label{fcteq} 
\mc{S}(s)^{-1}=\mc{S}(1-s), \quad \mc{P}_{\la}=\frac{\Gamma(-\la)}{\Gamma(\la+1)}\mc{P}_{-\la-1}\mc{S}(\la+1).
\end{equation}
The operator $\mc{S}(s)$ is an elliptic pseudo-differential operator of (complex) 
order $2s-1$ on $\mathbb{S}^1$, with principal symbol that of $\Delta_{\mathbb{S}^1}^{2s-1}$. 
This follows from the formula above but also in a more general setting by the 
works \cite{JoSa,GrZw}. 
We remark that for $k\in\nn$, the operator $\mc{S}(1/2+k)$ is a differential 
operator of order $2k$ (matching with the analysis of \cite{GrZw}), and it is invertible from the expression \eqref{formulaS(s)}.

We get for  $\la\notin (-\demi-\zz)\cup -\nn$
\begin{equation}\label{Flapm} 
\begin{split}
F_\la^-(0)=\pi^{1/2}2^{\la}\frac{\Gamma(\la+\demi)}{\Gamma(\la+1)}
\cjg \omega, \chi\cjd_{\mathbb{S}^1}, \quad
 F_\la^+(0)=\pi^{1/2}2^{-\la-1}\frac{\Gamma(-\la-\demi)}{\Gamma(\la+1)}\cjg \mc{S}(\la+1)\omega,\chi\cjd_{\mathbb{S}^1}.
 \end{split}\end{equation}
At $\la=-1/2+k$ with $k\in\nn$ 
there is a pole of order $1$ in the expression of $F_\la^+(0)$ and it follows from 
\cite{GrZw} that, with the notation of \eqref{weakasymptotic0},  
\[
F_{-1/2+k}^-(0)=\frac{\pi^{1/2}2^{-1/2+k}\Gamma(k)}{\Gamma(k+1/2)}
\cjg \omega, \chi\cjd_{\mathbb{S}^1}, \quad F_{-1/2+k}^+(0)=c_k\cjg\mc{S}(1/2+k)\omega,\chi\cjd
\] 
for some $c_k\not=0$.
For $\la=-1/2$, $F_\la^+(0)$, there is a constant $c_0\in\rr$ so that
\begin{equation}\label{Sen1/2} 
F_{-1/2}^+(0)=-\sqrt{2}\cjg \omega,\chi\cjd, \quad F_{-1/2}^-(0)=\sqrt{2}\cjg \pl_\la\mc{S}(1/2)\omega,\chi\cjd
+c_0\cjg \omega,\chi\cjd\end{equation}
and $\pl_\la\mc{S}(1/2)=\log(\Delta_{\mathbb{S}^1})+A$ for some pseudo-differential operator $A$ of order $0$ on $\mathbb{S}^1$.
To deal with the case $\la\in -\demi-\nn$, we use the functional equation \eqref{fcteq}
of the scattering operator of $\hh^{2}$: we deduce that for $\la=-1/2-k$ with $k\in\nn$, 
\begin{equation}\label{-1/2-kbis}
F^-_{-1/2-k}(0)=c'_k\cjg \omega,\chi\cjd_{\mathbb{S}^1}, \quad F^+_{-1/2-k}(0)=c''_k\cjg \mc{S}(1/2-k)\omega,\chi\cjd\end{equation} 
for some $c'_k\not=0$ and $c_k''\not=0$. This gives the expression for the inverse of $\mc{P}_{\la}$ at those points. 

To conclude, we discuss the range and kernel of $\mc{P}_{-n}$ if $n\in\nn$. Using the complex coordinate 
$z\in\cc$ for the ball model of $\hh^2$, this operator is 
\[\begin{split} 
\mc{P}_{-n}(\omega) (z)= &(1-|z|^2)^{1-n}\int_{\mathbb{S}^1}\omega(\nu)|z\bar{\nu}-1|^{2(n-1)}d\nu\\
=& (1-|z|^2)^{1-n}\int_{0}^{2\pi}\omega(e^{i\alpha})(|z|^2+1-\bar{z}e^{i\alpha}-ze^{-i\alpha})^{(n-1)}d\alpha.
\end{split}\]
From this we deduce that the range of  $\mc{P}_{-n}$ is finite and its kernel contains the space
\[W_n:=\{\omega\in \mc{D}'(\mathbb{S}^1); \cjg \omega, e^{ik\alpha}\cjd=0, \forall k \in\zz\cap[-n+1,n-1]\}.\]
In fact, from the second functional equation \eqref{fcteq} and the formula \eqref{formulaS(s)}, we see that 
\begin{equation}\label{kerP-n}
\ker \mc{P}_{-n}=\ker( {\rm Res}_{1-n}\mc{S}(\la))=W_n
\end{equation}

\subsection{Co-compact and convex co-compact quotients} 
\label{sec:quotients}

Below, we will consider two types of hyperbolic surfaces, the compact and the 
convex co-compact ones. Consider a discrete subgroup $\Gamma\subset G$ containing 
only hyperbolic transformations, i.e. transformations fixing two points in 
$\bbar{\hh}^2$. The group $\Gamma$ acts properly discontinuously 
on $\hh^2$ and the quotient $M=\Gamma\backslash \hh^2$ is a smooth oriented 
hyperbolic surface. We say that $\Gamma$ is \emph{co-compact} if $M$ is compact.

Denote by $\Lambda_\Gamma\subset \mathbb{S}^1$ 
the limit set of the group $\Gamma$, i.e. the set of accumulation points of the 
orbit $\Gamma.0\in \hh^{2}$ of $0\in \hh^2$ on $\mathbb{S}^1=\pl\hh^{2}$. 
We will call $\Omega_\Gamma=\mathbb{S}^1\setminus \Lambda_\Gamma$ the set of 
discontinuity of $\Gamma$, on which $\Gamma$ acts properly discontinuously.  

If $\Gamma$ is co-compact, then $\Lambda_\Gamma=\mathbb{S}^1$.
The subgroup $\Gamma$ is called \emph{convex co-compact}, 
if it is not co-compact and it the action of $\Gamma$
on the convex hull ${\rm CH}(\Lambda_\Gamma)\subset\hh^2$ of the limit set 
$\Lambda_\Gamma$ in $\bbar{\hh}^{2}$ is co-compact, that is 
$\Gamma\backslash {\rm CH}(\Lambda_\Gamma)$ is compact (see e.g. 
\cite[Section 2.4]{Bo}). 
In this case the group $\Gamma$ acts totally discontinuously, freely, on 
$\hh^{2}$ and more generally on $\hh^{2}\cup \Omega_\Gamma=\bbar{\hh}^{2}\setminus \Lambda_\Gamma$.
The manifold $M=\Gamma\backslash \hh^{2}$ is complete with infinite volume, 
and it is the interior of a smooth compact manifold with boundary 
$\bbar{M}:=\Gamma\backslash (\hh^{2}\cup \Omega_\Gamma)$. Here we notice that 
$\hh^{2}\cup \Omega_\Gamma$ is also a smooth manifold with boundary but it is non-compact. 
The boundary $\pl\bbar{M}:=\Gamma\backslash \Omega_\Gamma$ of $\bbar{M}$ is compact. 

We now consider $M=\Gamma\backslash \hh^2$ which is either compact or convex 
co-compact (here $M$ could be as well the whole $\hh^2$). 
The unit tangent bundle bundle of $M$ is  $SM=\Gamma\backslash S\hh^{2}\simeq \Gamma\backslash G$, 
and we let $\pi_{\Gamma}: S\hh^{2}\to SM$ be the induced covering map. The 
geodesic flow $\varphi_t:SM\to SM$ on $SM$ lifts to the geodesic flow on $S\hh^2$, 
the left invariant vector fields $X,U_\pm,X_\perp,V$ on $T(S\hh^2)\simeq TG$ 
descend to $SM$ via $d\pi_\Gamma$; we will keep the notation $X$ instead of 
$d\pi_\Gamma.X$, and similarly for the vector fields $U_\pm,X_\perp,V$.
The flow $\varphi_t$ is generated by the vector field $X$ and there is an Anosov 
flow-invariant smooth splitting   
\begin{equation}\label{anosov} 
T(SM)=\rr X\oplus E_s\oplus E_u
\end{equation}
where $E_u=\rr U_-$, $E_s=\rr U_+$ are the stable and unstable bundles satisfying 
the condition \eqref{stableunstable}. Using the Anosov splitting \eqref{anosov},  
we define the subbundles $E_0^*$, $E_s^*$ and $E_u^*$  of $T^*(SM)$ by
\begin{equation}\label{dualspl}
E_u^*(E_u\oplus \rr X)=0 ,\quad E_s^*(E_s\oplus \rr X)=0, \quad E_0^*(E_u\oplus E_s)=0.
\end{equation}
\subsection{Complex line bundles}\label{fouriermodes}
Note that $M=\Gamma\backslash \hh^2$ carries a complex structure so that $\pi_\Gamma: \hh^2\to M$ is holomorphic and that we can thus
consider the complex line bundles $\mc{K}:=(T^*M)^{1,0}$ and $\mc{K}^{-1}:=(T^*M)^{0,1}$.
Let us consider their tensor powers: for each $k\in\zz$, set 
$\mc{K}^k:=\otimes^{|k|}\mc{K}^{{\rm sign}(k)}$.
The bundles $\mc{K}^k$  are holomorphic line bundles over $M$, which in addition 
are trivial when $M$ is convex co-compact (see \cite[Theorems 30.1 and Th 30.4]{Fo}). A section $C_c^\infty(M; \mc{K}^k)$ 
can be viewed as a function in $C_c^\infty(SM)$ by the map 
\[ 
\pi_k^*: C_c^\infty(M; \mc{K}^k)\to C_c^\infty(SM), \quad \pi_k^*u(x,v):=u(x)(\otimes^k v).
\]
We denote by ${\pi_k}_*: \mc{D}'(SM)\to \mc{D}'(M; \mc{K}^k)$ its transpose 
defined by duality. Note that the operator $\pi_k^*$ extends to $\mc{D}'(M; \mc{K}^k)$ 
and $(2\pi)^{-1}{\pi_k}_*\pi_k^*$ is the identity map on 
$\mc{D}'(M; \mc{K}^k)$.
Each smooth function $f\in C_c^\infty(SM)$ can be decomposed into Fourier modes in the fibers of 
$SM$ by using the eigenvectors of the vector field $V$:
\[ 
f=\sum_{k\in\zz} f_k\, \quad \textrm{ with }Vf_k=ikf_k \textrm{ and }f_k=\frac{1}{2\pi}\pi_k^*{\pi_k}_*f.
\]
It is easy to see that for each $f\in C^\infty_c(SM)$, $s\geq 0,$ and
$N>0$
\begin{equation}\label{borneuk} 
||f_k||_{H^s(SM)}\leq C_{f,N,s}\cjg k\cjd^{-N}
\end{equation}
for some constant $C_{f,N,s}$ independent of  $k$. A distribution $u\in \mc{D}'(SM)$ 
can also be decomposed as a sum 
\[ 
u=\sum_{k\in\zz} u_k\, \quad \textrm{ with }Vu_k=iku_k, \quad u_k=\frac{1}{2\pi}\pi_k^*{\pi_k}_*u
\]
which converges in the distribution sense. In order to see this recall that any 
distribution $u\in \mathcal D'(SM)$ restricted to a precompact open set $A\subset SM$ 
is of finite order, i.e.  there is $s>0$, $C>0$ with
\begin{equation}\label{finiteorder}
|\cjg u,\varphi\cjd| \leq C\|\varphi\|_{H^s(A)}, ~\forall~\varphi\in C^\infty_c(A).
\end{equation}
Now for $f\in C_c^\infty(SM)$, we can 
write $f=\sum_{k\in\zz} f_k$. Then (\ref{borneuk}) and (\ref{finiteorder})
imply that 
\[
\cjg u,f\cjd=\sum_{k\in\zz} \cjg u,f_k\cjd=\frac{1}{2\pi}\sum_{k\in\zz} \cjg \pi_k^*{\pi_k}_*u,f\cjd
\]
is absolutely convergent. For convenience of notations, we avoid the $\pi_k^*,{\pi_k}_*$ 
operators and we will view $u_k$ as an element in $\mc{D}'(M; \mc{K}^k)$ or as an element
in $\{w\in \mathcal D'(SM) \text{ with }Vw=ikw\}$ depending on which point of view
is more appropriate in a given situation.
First of all $V$ acts on $\mathcal D'(M,\mathcal K^k)$ by
\[
 V: \mathcal D'(M,\mathcal K^k)\to \mathcal D'(M,\mathcal K^k),~~u_k\mapsto V u_k=ik u_k.
\]
Furthermore, if we define the complex valued vector fields 
\begin{equation}\label{etapm}
\eta_\pm :=\demi(X\pm iX_\perp),
\end{equation}
they fulfill the commutation relations
\[
 [V,\eta_\pm]  = \pm\eta_\pm.
\]
They are called the raising/lowering operators as they shift the vertical Fourier components by $\pm 1$ and
when restricted to sections of $\mc{K}^k$ through $\pi_k^*$, they define operators 
\begin{equation}\label{etapmprop}
 \eta_\pm: \mc{D}'(M; \mc{K}^k)\mapsto \mc{D}'(M;\mc{K}^{k\pm 1}).
\end{equation}
If $z=x+iy$ are local isothermal 
coordinates, the hyperbolic metric can be written as $g=e^{2\alpha(z)}|dz|^2$
and for $k\geq 0$ the operators $\eta_\pm$ acting on a section $u\in C_c^\infty(M; \mc{K}^k)$ 
of the form $u=f(z)dz^k$  is given by 
\begin{equation}\label{formulaeta}
\eta_-u = e^{-2\alpha}(\pl_{\bar{z}}f)dz^{k-1}, \quad 
\eta_+u = (e^{2k\alpha}\pl_{z}(e^{-2k\alpha}f))dz^{k+1}.
\end{equation}
A similar expression holds for $k\leq 0$ and we directly deduce 
\begin{equation}\label{dbar}
\forall k\geq 0,\,\,  \eta_-u_k=0 \iff \bbar{\pl}u_k=0 , \quad \forall k<0,\,\,  \eta_+u_k=0\iff {\pl}u_k=0.
\end{equation}
We notice that these operators $\eta_\pm$, as well as the operators $V,X$ and $X_\perp=[X,V]$, have nothing to do with constant curvature and Lie groups, they are well defined for any oriented compact Riemannnian surface, see Guillemin-Kazhdan \cite{GuKa}. 
The \emph{Casimir operator}  is defined as the second order operator on $SM$ by 
\begin{equation}\label{Casimir} \begin{split}
\Omega=& X^2+X_\perp^2-V^2= X^2-V^2+(2u_-+V)^2=X^2+4U_-^2+2U_-V+2VU_-\\
=& X^2-X+4U_-+4VU_-.
\end{split}\end{equation}
It satisfies $(\Omega u)_k=\Omega u_k$ since $\Omega$ commutes with $V$.

For later purpose, we will need a few lemmas which follow for algebraic reasons 
and Fourier decomposition in the fibers. 
\begin{lemm}\label{recurrence_relation}
Let $\lambda\in \cc$ and $u,v\in \mathcal D'(SM)$ be two distributions, then 
they fulfill the set of differential equations
 \[
  (X+\lambda)u = v \text{ and }U_-u=0
 \]
if and only if their Fourier modes fulfill the recursion relations
\begin{equation}\label{recursion}
 2\eta_\pm u_{k\mp1} = (-\lambda \mp k)u_k + v_k.
\end{equation}
\end{lemm}
\begin{proof}
Recall that $X=\eta_+ + \eta_-$ and and assume $(X+\lambda)u = v$
then, taking the Fourier components of this equation and using \eqref{etapmprop}
we get for any $k\in \zz$
\begin{equation}\label{eq1}
 \eta_+u_{k-1} + \eta_-u_{k+1} +\lambda u_k= v_k.
\end{equation}
Similarly we can express $U_- = -i(\eta_+-\eta_-) - V$ and the condition $U_-u = 0$
becomes
\begin{equation}\label{eq2}
-i(\eta_+u_{k-1} - \eta_-u_{k+1}) - ik u_k = 0 ~ \forall k\in \zz.
\end{equation}
Now inserting one equation into the other, one deduces that \eqref{eq1}
and \eqref{eq2} are equivalent to \eqref{recursion} and this completes the proof.
\end{proof}
\begin{lemm}\label{u0=0}
Let $\la\in \cc$ and $u\in \mc{D}'(SM)$ satisfies $U_-u=0$, $(X+\la)^ju=0$ 
for some $j\geq 1$. Set $u^{(\ell)}:=(X+\la)^\ell u$ for each $\ell\leq j$ and assume that $u^{(j-1)}\not=0$.
Then for each $\ell\leq j-2$
\[
(\Delta_M+\la(\la+1))u_0^{(j-1)}=0, \quad (\Delta_M+\la(\la+1))^{j-1-\ell}u_0^{(\ell)}=(2\la+1)^{j-1-\ell}u^{(j-1)}_0.
\]
\end{lemm}
\begin{proof}
Denote by $u^{(\ell)}:=(X+\la)^\ell u$, then $U_-u^{(\ell)}=0$ for all $\ell$ by the fact that $X$ preserves $\ker U_-$. 
Using \eqref{Casimir} we get for each $\ell\leq j$
\[\begin{split} 
\Delta_Mu^{(\ell)}_0 &=-\Omega u^{(\ell)}_0=((X-X^2)u^{(\ell)})_0=((X-1)(\la u^{(\ell)}-u^{(\ell+1)}))_0\\
& =-\la(\la+1)u^{(\ell)}_0+ (2\la+1) u^{(\ell+1)}_0-u^{(\ell+2)}_0
\end{split}\]
The result then follows from an easy induction.
\end{proof}
\begin{prop}\label{no_int_jord}
Let $n\in\nn$, assume that $u\in \mc{D}'(SM)$ satisfies  $U_-u=0$, $u_0=0$ and $(X-n)^ju=0$ for some $j\in\nn$,  then $(X-n)u=0$.
\end{prop}
\begin{proof} Let $u^{(\ell)}:=(X-n)^\ell u$. By Lemma \ref{u0=0}, we have $u^{\ell}_0=0$ for all 
$\ell=0,\dots,j-1$.
Assume that $(X-n)u\not=0$, then without loss of generality we can assume that $u^{(j-1)}\not=0$. 
Then using that $(X-n)$ preserves $\ker U_-$, we get that
$w:=u^{(j-2)}$  and $v:=u^{(j-1)}$ are non-zero distributions such that $U_- w = U_-v=0$, $(X-n)w=v$, $(X-n)v=0$ as well as $w_0=v_0$. 
Let us next use the knowledge that $w_0=v_0=0$ in order to obtain a contradiction.
 Lemma~\ref{recurrence_relation} applied to
$(X-n)w=v$ and $(X-n)v=0$ implies, that 
\begin{equation}
 \label{u_recursion} 
 2\eta_\pm w_{k\mp 1} = (n\mp k) w_k+ v_k, \quad 2\eta_\pm v_{k\mp 1} = (n\mp k) v_k
\end{equation}

From $v_0=0$ and the second equation of \eqref{u_recursion} we obtain that $v_k=0$ for all $|k|<n$.
Now this knowledge together with $w_0=0$ and the first recursion relation in \eqref{u_recursion}
leads to $w_k = 0$ for all $|k|<n$. Now let us consider the first equation of \eqref{u_recursion} with $k=n$
\[ 
 0 = 2\eta_+ w_{n-1} = v_n.
\]
Using once more the second recursion relations \eqref{u_recursion} this implies
$v_k=0$ for all $k\geq n$. Analogously we obtain $v_{-n} = 0$ and using the recursion
we see that $v_k=0$ for all $k\neq 0$. Thus $v=0$ which is the desired contradiction.
\end{proof}

\section{Ruelle resonances for co-compact quotients}

In this section, we consider a co-compact discrete subgroup $\Gamma\subset G$ acting freely 
on $\hh^{2}$, so that $M:=\Gamma\backslash \hh^{2}$ is a smooth oriented compact hyperbolic surface, 
and we describe the Ruelle resonance spectrum and eigenfunctions. 
The characterization of the  spectrum and eigenfunctions was done in \cite{DFG} except for some 
special points localized at negative half-integers. Here we analyze those points as well, and we simplify 
the proof of the fact that the algebraic multiplicities and geometric multiplicities agree.

First we recall the result \cite[Theorems 1.4 and 1.7]{FaSj} in the case of geodesic flows(see also  \cite[Theorem 1]{BuLi} or \cite[Propositions 3.1 and 3.3]{DyZw}):
\begin{prop}\label{resolventcompact}
Let $(M,g)$ be a Riemannian manifold with Anosov geodesic flow. 
For each $N>0$, the vector field $X$ generating the flow is such that $\la\mapsto -X-\la$
is a holomorphic family of Fredholm operators of index $0$
\[-X-\la : {\rm Dom}(X)\cap \mc{H}^N\to \mc{H}^N \] 
for $\{{\rm Re}(\la)>-N\}$, where $\mc{H}^N$ is a Hilbert space such that  
$C^\infty(SM)\subset \mc{H}^N\subset \mc{D}'(SM)$. 
The spectrum of $-X$ on $\mc{H}^N$ in $\{{\rm Re}(\la)>-N\}$ is discrete, 
contained in $\{{\rm Re}(\la)\leq 0\}$, and minus the residue 
\[\Pi^X_{\la_0}:= -{\rm Res}_{\la_0}(-X-\la)^{-1},  \quad {\rm Re}(\la_0)>-N\]
is a projector onto the finite dimensional space 
\[ {\rm Res}_X(\la_0):=\{ u\in \mc{H}^N ; \exists j\in\nn, (-X-\la_0)^ju=0\}\] 
of generalized eigenstates, satisfying $X\Pi^X_{\la_0}=\Pi^X_{\la_0}X$. 
The eigenvalues, generalized eigenstates and the Schwartz kernel of 
$\Pi^X_{\la_0}$ are independent of $N$ and one has 
\begin{equation}\label{eq:ruelle_res_WF_cond}
{\rm Res}_X(\la_0)=\{ u\in \mc{D}'(SM) ; {\rm WF}(u)\subset E_u^*,\, \exists j\in\nn,  (-X-\la_0)^ju=0\}.
\end{equation}
\end{prop}
Notice that the description \eqref{eq:ruelle_res_WF_cond} of resonant states is not explicit in \cite{FaSj} but if follows easily from that paper \cite{FaSj} - a proof is given in Lemma 5.1 of \cite{DFG}.
The eigenvalues, eigenstates and generalized eigenstates are respectively 
called \emph{Ruelle resonances},  \emph{Ruelle resonant states} and 
\emph{Ruelle generalized resonant states}. The existence 
of  generalized resonant states which are not resonant states means that the 
algebraic multiplicity of the eigenvalue is larger than the geometric 
multiplicity, in which case there are Jordan blocks in the matrix 
representing $(-X-\la_0)$ on ${\rm Im}(\Pi^X_{\la_0})$. 

A direct corollary of Proposition \ref{resolventcompact} is that the $L^2$-resolvent 
\[\la\mapsto R_X(\la):=(-X-\la)^{-1}\] 
of $X$ extends from ${\rm Re}(\la)>0$ to $\cc$ as a meromorphic 
family of bounded operators $C^\infty(SM)\to \mc{D}'(SM)$, with poles the Ruelle 
resonances of $-X$.  Indeed, Proposition 3.2 in \cite{DyZw} shows that in the region ${\rm Re}(\la)>C$ for some $C>0$, the $L^2$-resolvent $R_X(\la)$ is also bounded as a map 
$\mc{H}^N\to \mc{H}^N$, thus $R_X(\la)$ is the inverse of $-X-\la$ on $\mc{H}^N$ in that half-space. By Proposition \ref{resolventcompact},  $R_X(\la):\mc{H}^N\to \mc{H}^N$ extends as a meromorphic family of bounded operators 
in ${\rm Re}(\la)>-N$ and thus is meromorphic as a bounded map 
$C^\infty(SM)\to \mc{D}'(SM)$. Density of $C^\infty(SM)$ in $\mc{H}^N$ and 
unique continuation in $\la$ shows that $R_X(\la)$ is actually independent of 
$N$ as a map $C^\infty(SM)\to \mc{D}'(SM)$. The fact that 

We define the following spaces for $j\geq 1$
\begin{equation}\label{Resk}
{\rm Res}^j_X(\la_0)=\{ u\in \mc{D}'(SM) ; {\rm WF}(u)\subset E_u^*,\,  (-X-\la_0)^ju=0\},
\end{equation}
The operator $(-X-\la_0)$ is nilpotent on the finite dimensional space 
${\rm Res}_X(\la_0)$ and 
$\la_0$ is a Ruelle resonance if and only if ${\rm Res}_X^1(\la_0)\not=0$. The 
presence of Jordan blocks for $\la_0$ is equivalent to having 
${\rm Res}_X^j(\la_0)\not={\rm Res}_X^1(\la_0) $ for some $j>1$. Let us define 
for $j\geq1, m\geq0$ the subspace 
\begin{equation}\label{Vm}
V_m^j(\la_0):=\{ u\in {\rm Res}_X^j(\la_0); U_-^{m+1}u=0\}.
\end{equation}
Obviously $V_m^j(\la_0)\subset V_{m+1}^j(\la_0)$ for each $m\geq 0,j\geq 1$. 
The spaces $V_0^j(\la_0)$ are spanned by the generalized resonant states which 
are invariant under the unstable horocycle flow.
First, following \cite{DFG}, we have the 
\begin{lemm}\label{decomp}
For each $\la_0\in \cc$, there exists $m\geq 0$ such that
${\rm Res}_X^1(\la_0)=V_m^1(\la_0)$.
\end{lemm}
\begin{proof}
By the commutation relation \eqref{commutation}, for $u\in {\rm Res}_X^1(\la_0)$ 
we have for each $m\geq 0$
\[ (-X-\la_0-m)U^m_-u=0\] 
thus for $m>-{\rm Re}(\la_0)$, we get $U_-^mu=0$ since we know there is no Ruelle 
resonance in 
${\rm Re}(\la)>0$ by Proposition \ref{resolventcompact} and ${\rm WF}(U_-^mu)\subset {\rm WF}(u)\subset E_u^*$.
\end{proof}
Next we analyze the generalized resonant states that are in $\ker U_-$, i.e. 
the spaces $V_0^j(\la_0)$ for $j\geq 1$. This part is essentially contained in 
\cite{FlFo} (even though they 
do not consider the problem from the point of view of Ruelle resonances) but we 
provide a more geometric method by using the Poisson operator, 
with the advantage that this approach extends to the convex co-compact setting. Our proof does not use the representation theory of ${\rm SL}_2(\rr)$ at all.
\begin{theo} \label{noninteger}
Let $M=\Gamma\backslash \hh^{2}$ be a smooth oriented compact hyperbolic surface 
and let $SM$ be its unit tangent bundle.\\ 
1) For each $\la_0\in \cc\setminus (-\demi\cup  -\nn)$  the pushforward map 
${\pi_0}_*: \mc{D}'(SM)\to \mc{D}'(M)$ restricts to a linear isomorphism of complex vector spaces
\begin{equation}\label{pi*iso_cpt}
{\pi_0}_* : V_0^1(\la_0) \to \ker (\Delta_M+\la_0(1+\la_0)) 
\end{equation}
where $\Delta_M$ is the Laplacian on $M$ acting on functions, and there are no 
Jordan blocks,  
i.e. $V_0^j(\la_0)=V_0^1(\lambda_0)$ for $j>1$.\\
2) For $\la_0=-\demi$, the Jordan blocks are of order $1$, i.e. $V_0^j(-\demi)=V_0^1(-\demi)$ 
for $j>2$, 
the map 
\begin{equation}\label{pi*iso2_cpt} 
{\pi_0}_* : V_0^1(-\demi)\to \ker (\Delta_M-\tfrac{1}{4})
\end{equation}
is a linear isomorphism of complex vector spaces and  
\begin{equation}
 \label{eq:pi_projection_cpt}
 {\pi_0}_* : V_0^2(-\demi) \to \ker (\Delta_M-\tfrac{1}{4})
\end{equation}
has a kernel of dimension $\dim \ker (\Delta_M-\tfrac{1}{4})$.\\
3) For $\la_0=-n\in -\nn$, there are no Jordan blocks, 
i.e. $V_0^j(-n)=V_0^{1}(-n)$ if $j>1$. The following map is an isomorphism of complex vector spaces
\begin{equation}\label{pi*iso3_cpt}
 {\pi_n}_*\oplus {\pi_{-n}}_* : V_0^1(-n)\to H_{n}(M)\oplus H_{-n}(M)
\end{equation}
with $H_n(M):=\{u\in C^\infty(M; \mc{K}^n); \bbar{\pl}u=0\}$, 
$H_{-n}(M):=\{u\in C^\infty(M; \mc{K}^{-n}); \pl u=0\}$. 
\end{theo}
\begin{proof} \textbf{Case $\la_0\not\in -1/2\cup -\nn$, injectivity}. 
Since $X+\la_0$ is nilpotent on $V_0^2(\la_0)$, we can decompose this space into 
Jordan blocks and there is a non-trivial Jordan block if and only if $X+\la_0$ is 
not identically $0$ on $V_0^2(\la_0)$.
Assume $\la_0$ is a Ruelle resonance and let $u^{(0)}\in V^1_0(\lambda_0)\setminus\{0\}$
be a resonant state. If there is a Jordan block
associated to $u^{(0)}$, there is $u^{(1)}\in V^2_0(\la_0)$ so that 
\[ (X+\la_0)u^{(0)}=0 ,\quad (X+\la_0)u^{(1)}=u^{(0)}.\]  
We will show that in fact this is not possible. 
The wave front set of $u^{(j)}$ is contained in $E_u^*$.
We lift $u^{(0)}$ and $u^{(1)}$ on $S\hh^{2}$ by $\pi_{\Gamma}:S\hh^{2}\to SM$ 
to $\til{u}^{(0)},\til{u}^{(1)}$ so that $\gamma^*\til{u}^{(j)}=\til{u}^{(j)}$ 
for all $\gamma\in \Gamma$, $j=0,1$. 
Take the distribution $v^{(0)}:=\Phi_-^{-\la_0}\til{u}^{(0)}$ on $S\hh^2$ where 
$\Phi_-$ is defined by \eqref{phipm}. It satisfies $Xv^{(0)}=0$ as $X\Phi_-=-\Phi_-$.  
Then by \eqref{Qpmiso} there exists $\omega^{(0)}\in \mc{D}'(\mathbb{S}^1)$ so 
that $\mc{Q}_-\omega^{(0)}=v^{(0)}$.
Using $\gamma^*\tilde u^{(0)} = \tilde u^{(0)}$ together with \eqref{changeofPhi}, we have 
that for any $\gamma\in \Gamma$
\[
\gamma^*\omega^{(0)}=N_\gamma^{-\la_0}\omega^{(0)} \textrm{ with } N_\gamma:=|d\gamma|^{-1}.
\]
Next we compute, using \eqref{Upmkills}, that for $v^{(1)}:=\Phi_-^{-\la_0}(\til{u}^{(1)}+\log(\Phi_-)\til{u}^{(0)})$, 
\[
Xv^{(1)}=0, \quad U_-v^{(1)}=0
\]
and thus by \eqref{Qpmiso}  there is $\omega^{(1)}\in \mc{D}'(\mathbb{S}^1)$ 
such that $\mc{Q}_-\omega^{(1)}=v^{(1)}$. Since $\gamma^*\til{u}^{(1)}=\til{u}^{(1)}$ 
for all $\gamma\in \Gamma$, we get 
\begin{equation}\label{covariance}
\gamma^*v^{(1)}-N_\gamma^{-\la_0}v^{(1)}=N_{\gamma}^{-\la_0}\log(N_\gamma)v^{(0)}, \quad 
\gamma^*\omega^{(1)}-N_\gamma^{-\la_0}\omega^{(1)}=N_{\gamma}^{-\la_0}\log(N_\gamma)\omega^{(0)}.
\end{equation}
We apply the Poisson operator $\mc{P}_{\la_0}$ to $\omega^{(0)}$ and $\omega^{(1)}$, 
where $\mc{P}_\la$
is defined in Lemma \ref{poissoniso}.
Using the same lemma, we get that $\til{\varphi}_0:=\mc{P}_{\la_0}(\omega^{(0)})={\pi_0}_* \til{u}^{(0)}\not=0$ 
is $\Gamma$-invariant and satisfies
\[
(\Delta_{\hh^2}+\la_0(1+\la_0))\til{\varphi}_0=0 
\]
on $\hh^{2}$, and thus descend to $\varphi_0:={\pi_0}_*u^{(0)}$ on $M$
as a non-zero eigenfunction of $\Delta_{M}$ with eigenvalue $\la_0(1+\la_0)$ and this has 
to be a smooth function by ellipticity.  
Since $\Delta_M$ is  self-adjoint on $L^2(M)$ this implies $\lambda_0(1+\lambda_0)\in \rr$.

Now set $\psi:=\mc{P}_{\la_0}(\omega^{(1)})\in C^\infty(\hh^{2})$,
we also get from \eqref{changeofPhi}, \eqref{covariance}
\[ 
\gamma^*\psi-\psi= \mc{P}_{\la_0}(\log(N_\gamma)\omega^{(0)}), \quad (\Delta+\la_0(1+\la_0))\psi=0.
\]
Now, for arbitrary $\omega\in\mc{D}'(\mathbb{S}^1)$, we make the observation, by Taylor 
expanding the equation $(\Delta_{\hh^2}+\la(1+\la))\mc{P}_\la(\omega)=0$ with 
respect to $\la$ at $\la_0$ that for $k\geq 1$
\begin{equation}\label{deriveepoisson}
\begin{gathered}
(\Delta_{\hh^2}+\la_0(1+\la_0))\frac{\pl_{\la}^k\mc{P}_{\la_0}(\omega)}{k !} +(2\la_0+1)\frac{\pl_{\la}^{k-1}\mc{P}_{\la_0}(\omega)}{(k-1)!}+\frac{\pl_{\la}^{k-2}\mc{P}_{\la_0}(\omega)}{(k-2)!}=0, \\
\pl_\la^k \mc{P}_\la(\omega)(x)= {\pi_0}_*\left(\Phi_-^{\la}(\log \Phi_-)^{k}
\mc{Q}_-(\omega)\right),
\end{gathered}
\end{equation}
and for each $\gamma\in \Gamma$, $\gamma^*(\pl_{\la}\mc{P}_{\la_0}(\omega^{(0)}))=
\pl_{\la}\mc{P}_{\la_0}(\omega^{(0)})+\mc{P}_{\la_0}(\log(N_\gamma)\omega^{(0)})$.
Thus, since $\la_0\not=-\demi$, we can set 
\[\til{\varphi}_1:= -(1+2\la_0)^{-1}(\pl_{\la}\mc{P}_{\la_0}(\omega^{(0)})-\psi)\] 
and we get on $\hh^{2}$
\[ (\Delta_{\hh^2}+\la_0(1+\la_0))\til{\varphi}_1=\til{\varphi}_0, \;\; \textrm{ and } \forall \gamma\in \Gamma,\,\,  \gamma^*\til{\varphi}_1=\til{\varphi}_1.\] 
This means that $\til{\varphi}_1$ descends to a smooth function $\varphi_1$ on $M$, which satisfies 
the equation 
\begin{equation}\label{eq:laplace_Jordan}
(\Delta_M+\la_0(1+\la_0))\varphi_1=\varphi_0.
\end{equation}
In fact, since 
\[
\psi={\pi_0}_*(\Phi_-^{\la_0}\mc{Q}_-(\omega^{(1)})) = {\pi_0}_*(\til{u}^{(1)}+\log(\Phi_-)\til{u}^{(0)})= 
{\pi_0}_*(\til{u}^{(1)})+\pl_{\la}\mc{P}_{\la_0}(\omega^{(0)}), 
\]
we notice that $\varphi_1=(1+2\la_0)^{-1}{\pi_0}_*u^{(1)}$.
If $M$ is compact \eqref{eq:laplace_Jordan} can only hold if $\varphi_0=0$, 
since $\la_0(1-\la_0)\in\rr$ and thus 
\[ ||\varphi_0||_{L^2(M)}^2= \cjg(\Delta_M+\la_0(1+\la_0))\varphi_1,\varphi_0\cjd=0.
\]
But $\varphi_0=0$ implies $\omega^{(0)}=0$ by injectivity of $\mc{P}_{\la_0}$, and thus $u^{(0)}=0$, which leads to a contradiction and thus $u^{(1)}$ does not exist. 
We have thus proved the non-existence of Jordan blocks if $\la_0\not=-\demi$.  
It is also clear that, assuming now that there is no Jordan block, we 
can take the argument above and ignore the $u^{(1)},v^{(1)},\omega^{(1)},\til{\varphi}_1$ terms, it 
shows that $\varphi_0={\pi_0}_*u^{(0)}$ is in $\ker (\Delta_M+\la_0(1+\la_0))$ and 
is nonzero if $u^{(0)}\not=0$, 
and the map \eqref{pi*iso_cpt} is then injective.

\textbf{Case $\la_0\not\in -1/2\cup -\nn$, surjectivity.} To construct the reciprocal map, we proceed as follows: if 
$(\Delta_M+\la_0(1+\la_0))\varphi_0=0$ with $\varphi_0\not=0$, then by the 
surjectivity of $\mc{P}_{\la_0}$ in Lemma \ref{poissoniso}, there is 
$\omega^{(0)}\in \mc{D}'(\mathbb{S}^1)$
such that $\mc{P}_{\la_0}(\omega^{(0)})=\til{\varphi}_0$ where $\til{\varphi}_0$ 
is the lift of $\varphi_0$ to $\hh^{2}$. Using injectivity of $\mc{P}_{\la_0}$ and 
$\gamma^*\til{\varphi}_0=\til{\varphi}_0$ for all $\gamma\in \Gamma$, we get 
$\gamma^*\omega^{(0)}=N_\gamma^{-\la_0}\omega^{(0)}$. This implies that 
$\til{u}^{(0)}:=\Phi_-^{\la_0}\mc{Q}_-(\omega^{(0)})$ is $\Gamma$ invariant by 
\eqref{changeofPhi} and 
that $(X+\la_0)\til{u}^{(0)}=0$ and $U_-\til{u}^0=0$ by \eqref{Qpmiso}. Then 
$\til{u}^0$ descends to and element $u^0\in \mc{D}'(SM)$ satisfying $(X+\la_0)u^{(0)}=0$ 
and $U_-u^{(0)}=0$. Now the differential operator $U_-$ is elliptic outside 
$E_u^*\oplus E_0^*$ and $(X+\la_0)$ is elliptic outside 
$E_u^*\oplus E_s^*$,  thus ${\rm WF}(u^{(0)})\subset E_u^*$. Using the characterizaion
\eqref{eq:ruelle_res_WF_cond} this implies that 
$u^{(0)}$ is a Ruelle resonant state.

\textbf{Case $\la_0=-\demi$, Jordan block and kernel.} We will show that there can be a Jordan block of order $1$ but no 
Jordan blocks of order $2$. Assume there is a Jordan block of order $2$, i.e.  
\[
(X+\la_0)u^{(0)}=0 ,\quad (X+\la_0)u^{(1)}=u^{(0)}, \quad (X+\la_0)u^{(2)}=u^{(1)},
\]
for some $u^{(j)}\in V_0^{j+1}(\la_0)$. Define $\til{u}^{(j)}=\pi_{\Gamma}^*u^{(j)}$
their lifts to $S\hh^{2}$, let 
$v^{(0)}:=\Phi_-^{-\la_0}\til{u}^{(0)}$ and 
$v^{(1)}:=\Phi_-^{-\la_0}(\til{u}^{(1)}+\log(\Phi_-)\til{u}^{(0)})$ as before, and 
\[
v^{(2)}:= \Phi_-^{-\la_0}(\til{u}^{(2)}+\log(\Phi_-)\til{u}^{(1)}+\demi (\log \Phi_-)^2\til{u}^{(0)}).
\]
Then $Xv^{(j)}=0$, $U_-v^{(j)}=0$, and $\gamma^*v^{(j)}=N_\gamma^{-\la_0}\sum_{\ell=0}^j
\frac{(\log N_\gamma)^\ell}{\ell !}v^{(j-\ell)}
$ thus $v^{(j)}=\mc{Q}_-(\omega^{(j)})$ for some $\omega^{(j)}\in \mc{D}'(\mathbb{S}^1)$ satisfying 
$\gamma^*\omega^{(j)}=N_\gamma^{-\la_0}\sum_{\ell=0}^j\frac{(\log N_\gamma)^\ell}{\ell !}\omega^{(j-\ell)}$.
Set 
\[
\begin{gathered}
\til{\varphi}_0=\mc{P}_{\la_0}(\omega^{(0)})={\pi_0}_*\til{u}^{(0)}, \quad
\til{\varphi}_1:=\mc{P}_{\la_0}(\omega^{(1)})-\pl_{\la}\mc{P}_{\la_0}(\omega^{(0)})={\pi_0}_*\til{u}^{(1)},\\ 
\til{\varphi}_2:=\mc{P}_{\la_0}(\omega^{(2)})-\pl_{\la}\mc{P}_{\la_0}(\omega^{(1)})+\demi \pl_{\la}^2\mc{P}_{\la_0}(\omega^{(0)})={\pi_0}_*\til{u}^{(2)},\end{gathered}
\] 
then  using \eqref{deriveepoisson} we get $(\Delta_{\hh^2}-\tfrac{1}{4})\til{\varphi}_j=0$ for $j=0,1$ 
and $\gamma^*\til{\varphi}_j= \til{\varphi}_j$ for all $\gamma\in \Gamma$. As 
for the case $\la_0\not=-\demi$, 
$\til{\varphi}_j$ descend to smooth functions $\varphi_j$  on $M$ and, since 
$(\Delta_M-\tfrac{1}{4})\varphi_2=\varphi_0$, we see that $||\varphi_0||_{L^2(M)}=0$ 
if $\varphi_2$ is non-zero, leading to a contradiction. This shows that the Jordan 
block for $X$ at $\la_0=-\demi$ can be only of order $1$, and from the injectivity
of the Poisson-Helgason transformation we know that $\varphi_0\not=0$ is a non-zero 
element of $\ker (\Delta_M-\tfrac{1}{4})$. Thus the map \eqref{pi*iso2_cpt} is 
injective and the map \eqref{eq:pi_projection_cpt} has a kernel of dimension
less or equal to $\dim(\ker(\Delta_M-\tfrac{1}{4}))$.

\textbf{Case $\la_0=-\demi$, surjectivity.} Assume $\varphi^{(0)}\in \ker (\Delta_M-\tfrac{1}{4})$ is non zero, 
then by Lemma \ref{poissoniso} there is $\omega^{(0)}\in \mc{D}'(SM)$ such that
$\mc{P}_{-1/2}(\omega^{(0)})=\til{\varphi}_0$ where $\til{\varphi}_0$ is the lift 
of $\varphi_0$ on $\hh^{2}$. We use \eqref{Sen1/2} and  have the expansion as 
$t\to 0$ for each $\chi\in C^\infty(\mathbb{S}^1)$
\[
\int_{\mathbb{S}^1} \mc{P}_{-1/2}(\omega^{(0)})(\tfrac{2-t}{2+t}\nu)\chi(\nu)dS(\nu)=\sqrt{t}(-\sqrt{2}\log(t)\cjg \omega^{(0)},
\chi\cjd_{\mathbb{S}^1}+\cjg \omega^{(1)},\chi\cjd_{\mathbb{S}^1}+\mc{O}(t\log t))
\]
for some non-zero $\omega^{(1)}\in \mc{D}'(\mathbb{S}^1)$. Notice that if 
$t(x):=2\tfrac{1-|x|}{1+|x|}$ on $\hh^{2}$, 
we have $\gamma^*t(|x|\nu)=|d\gamma(\nu)|t(x)+\mc{O}(t(x)^2)$ as $|x|\to 1$ for 
each isometry $\gamma\in G$ and $\nu\in \mathbb{S}^1$.
Then, since  $\gamma^*\til{\varphi}_0=\til{\varphi}_0$ for each  $\gamma\in \Gamma$, 
we get 
\[
\gamma^* \omega^{(0)}=N_\gamma^{1/2}\omega^{(0)}, \quad 
\gamma^* \omega^{(1)}=N_\gamma^{1/2}\omega^{(1)}-\sqrt{2}N_\gamma^{1/2}(\log N_\gamma) \omega^{(0)}
\]
for each $\gamma\in \Gamma$.
Define $\til{u}^{(0)}:=\Phi_-^{-1/2}\mc{Q}_-(\omega^{(0)})$ and 
$\til{u}^{(1)}:=\Phi_-^{-1/2}\mc{Q}_-(\omega^{(1)})+\sqrt{2}(\log \Phi_-)\til{u}^{(0)}$, 
which are $\Gamma$-invariant and satisfy
\[
(X-\demi)\til{u}^{(0)}=0, \quad (X-\demi)\til{u}^{(1)}=\til{u}^{(0)}, \quad U_- \til{u}^{(0)}=U_-\til{u}^{(1)}=0. 
\] 
These distributions thus descend to $SM$ and are non-zero resonant and generalized 
resonant states since they also have wave-front sets in $E_u^*$ by the same argument 
as for $\la\not=-1/2$. In particular any element in $\ker(\Delta_M - \tfrac{1}{4})$
this produces a Jordan block of order $1$ and consequently the map \eqref{eq:pi_projection_cpt}
has a kernal of dimension $\dim\ker(\Delta_M-\tfrac{1}{4})$.

\textbf{Case $\la_0\in -\nn$, injectivity.} Let $\la_0=-n$ with $n\in\nn$.
First, if $u\in V_0^k(-n)$, we have 
$u_0\in \ker (\Delta_M+n(n-1))$. 
From the positivity of the Laplacian we get $u_0=0$ except possibly if $n=-1$ 
where $u_0$ must be constant. Let us show 
that in fact $u_0=0$ if $n=-1$: if $u\in V_0^1(-1)$ satisfies $u_0=1$, we 
have $2\eta_-u_1=1$ by \eqref{recursion}, but then 
\[0<||u_0||^2_{L^2}=2\cjg \eta_-u_1,u_0\cjd_{L^2}= -2\cjg u_1,\eta_+u_0\cjd_{L^2}=0
\]
leading to a contradiction. Thus $u_0=0$ for all $n$ if $u\in V_0^1(-n)$. Thus Proposition \ref{no_int_jord} implies that there are no Jordan 
Blocks for the Ruelle resonances. By the proof of Proposition \ref{no_int_jord}, the Fourier components $u_k$ of $u$ in the fibers satisfy $u_k=0$ for $|k|<n$.  By \eqref{recursion}
we have $\eta_-u_n=0$ and $\eta_+u_{-n}=0$, and thus by \eqref{dbar} we get 
$u_n\in H_n(M)$ and $u_{-n}\in H_{-n}(M)$. Furthermore note that $u_n=0$ implies
after iteratively applying \eqref{recursion}, that $u_k=0$ for all $k>n$, and similarly if $u_{-n}=0$.  The map \eqref{pi*iso3_cpt}
is thus injective.

\textbf{Case $\la_0\in -\nn$, surjecivity.}  Conversely, we want to prove that for each
$u_n\in H_n(M)$, there is a $u\in \mc{D}'(SM)$ so that 
$(X-n)u=0$ and $U_-u=0$. We construct $u$ as a formal sum $u=\sum_{k\geq n} u_k$ where 
$u_k$ are in $\pi_k^*(C^\infty(M;\mc{K}^k))$, that we will thus freely identify 
with sections of $\mc{K}^k$.
We set $u_k=0$ for all $k<n$ and we define recursively for $k>n$
\begin{equation}\label{recursive_def}
 u_{k} := \frac{2}{n-k}\eta_+ u_{k-1}
\end{equation}
and set $u_{-k}=\bbar{u_k}$. Clearly $u_k$ are smooth and in fact analytic since $u_n$ is. 
First let us show that the formal series $u=\sum_{k\in\zz} u_k$ fulfills 
$(X-n)u=0$ and $U_-u=0$.
According to Lemma~\ref{recurrence_relation} these conditions are equivalent 
to the fact that 
\begin{align}
2\eta_+ u_{k-1} &= (n- k)u_k \label{eq:raising_cond} \\
2\eta_- u_{k+1} &= (n+ k)u_k\label{eq:lowering_cond}
\end{align}
holds for all $k\in \zz$. We see, that (\ref{eq:raising_cond}) is already 
fulfilled for all $k\in \zz$ by our recursive definition of the Fourier modes 
via \eqref{recursive_def}. For $k < n-1$ also the condition 
\eqref{eq:lowering_cond} is fulfilled, as we have set all Fourier modes to be 
identically zero. For $k=n-1$ (\ref{eq:lowering_cond}) is true because we identified
$u_n$ with a holomorphic section in $\mc{K}^n$ and the case $k\geq n$
follows by induction from the following fact:
for each $\ell \in \nn_0$, if $2\eta_- u_{n+\ell} = (2n+\ell-1)u_{n+\ell-1}$ holds, then
 $2\eta_- u_{n+\ell+1} = (2n+\ell)u_{n+\ell}$.
This fact is a direct consequence of the commutation relation 
\begin{equation}
 \label{eta_pm_commut}
 [\eta_+,\eta_-] = -\frac{i}{2}V
\end{equation}
since
\[\begin{split}
 2\eta_-u_{n+\ell+1} &\underset{(\ref{recursive_def})}{=} -\frac{4}{\ell+1}\eta_-\eta_+ u_{n+\ell}
 \underset{(\ref{eta_pm_commut})}{=} -\frac{4}{\ell+1}(\eta_+\eta_- +\frac{i}{2} V) u_{n+\ell}\\
 &= -\frac{4}{\ell+1}\left(\frac{-\ell(2n+\ell-1)}{4} -\frac{(n+\ell)}{2}\right) u_{n+\ell}= (2n+\ell) u_{n+\ell}.
\end{split}\]
Next we need to show that the formal sum $u=\sum_k u_k$ defines a distribution, and it suffices to check 
that $||u_k||_{L^2}=\mc{O}( |k|^N)$ for some $N$ as $|k|\to \infty$.
Let us give an argument which is close to the approach of Flaminio-Forni \cite{FlFo}: 
let $k>n$ and consider
\[\begin{split}
 \|u_k\|^2_{L^2(SM)} =& \langle u_k,u_k\rangle_{L^2}=\left(\frac{2}{n-k}\right)\langle u_k,\eta_+u_{k-1}\rangle_{L^2} = -\left(\frac{2}{n-k}\right)\langle \eta_- u_k,u_{k-1}\rangle_{L^2}\\
 =& -\left(\frac{k+n-1}{n-k}\right)\langle u_{k-1},\eta_+u_{k-1}\rangle_{L^2}=\left(\frac{k+n-1}{k-n}\right)\|u_{k-1}\|^2_{L^2}.
\end{split}\]
Thus recursively we obtain for $\ell>0$ (here $\Gamma$ is the Euler Gamma function)
\[
 \|u_{n+\ell}\|^2_{L^2(SM)} = \Pi_{\ell,n}\|u_{n}\|^2_{L^2(SM)} \text{ where }\Pi_{\ell,n} = \frac{\Gamma(2n+\ell)}{\Gamma(\ell+1)\Gamma(2n)} =\prod_{r=1}^\ell \frac{2n-1 +r}{r}.
\]
Now it is direct to check that $\Pi_{\ell,n}=\mc{O}(\ell^{N})$ for some $N$. The same argument works with 
$u_{-n}$.
\end{proof}
Here notice that by Riemann-Roch theorem, the spaces $H_{n}(M)$ have complex dimension 
\begin{equation}\label{eq:dim_riemann_roch_cpt}
\dim H_{n}=\left\{\begin{array}{ll}
\demi(2n-1)|\chi(M)| & \textrm{ if }n>1\\
\demi |\chi(M)|+1 & \textrm{ if }n=1
\end{array}\right. 
\end{equation}
where $\chi(M)$ is the Euler characteristic of $M$.

To conclude this section, we describe the full Ruelle resonance spectrum of $X$ by using Theorem \ref{noninteger}.
\begin{corr}\label{otherbands}
Let $M=\Gamma\backslash \hh^{2}$ be a smooth oriented compact hyperbolic surface 
and let $SM$ be its unit tangent bundle. Then for each $\la_0$ with ${\rm Re}(\la_0)\leq 0$, 
$k\in\nn$, and $j\in\nn$, the operator $U_+^k$ is injective on $V_0^j(\la_0+k)$ if $\la_0+k\not=0$ and 
\[V_{k}^j(\la_0)=U_+^k(V_0^j(\la_0+k))\oplus V_{k-1}^j(\la_0).\]  
In other words, we get $V_{k}^j(\la_0)=\bigoplus_{\ell=0}^k U_+^\ell(V_0^j(\la_0+\ell)).$
\end{corr}
\begin{proof}
First, we have $U_+^k(V_0^j(\la_0+k))\subset V_{k}^j(\la_0)$ since 
\begin{equation}\label{relationk=0} 
(X+\la_0)^jU_+^k=U_+(X+\la_0+1)^jU_+^{k-1}=\dots=U^k_+(X+\la_0+k)^j
\end{equation}
and similarly $U_-^k$ maps $V_{k}^j(\la_0)$ to $V_0^j(\la_0+k)$ 
with kernel $V_{k-1}^j(\la_0)$, 
so that it remains to show that the map
\[ U_-^kU_+^k: V_0^j(\la_0+k)\to V_0^j(\la_0+k)\]
is one-to-one. Of course, since $V_0^j(\la_0+k)=0$ if ${\rm Re}(\la_0)+k>0$, it suffices to assume that 
${\rm Re}(\la_0)\leq -k$.
But a direct computation using $[U_+,U_-]=2X$ gives
\[ U_-^kU_+^k=k!\prod_{\ell=1}^k(-2X-k+\ell) \textrm{ on }V_0^j(\la_0+k).\]
When $j=1$, $X=-(\la_0+k){\rm Id}$ on $V_0^1(\la_0+k)$, thus
\[\prod_{\ell=1}^k(-2X-k+\ell)=\prod_{\ell=1}^k(2\la_0+k+\ell){\rm Id} \textrm{ on }V_0^1(\la_0+k).\]
This is invertible if $\la_0\not=-k$ since we know that ${\rm Re}(\la_0)+k\leq 0$. Now we can do an induction for 
$j>1$: assume that $\ker U_-^kU_+^k\cap V_0^{j-1}(\la_0+k)=0$, then 
using $(X+\la_0)U_-^kU_+^k=U_-^kU_+^k(X+\la_0)$, 
we see that if $u\in V_0^j(\la_0+k)\cap\ker U_-^kU_+^k$, 
$(X+\la_0)u\in V_0^{j-1}(\la_0+k)\cap \ker U_-^kU_+^k=0$ and this completes the argument when $\la_0+k\not=0$.  
\end{proof}
This Theorem describes the full Ruelle resonance spectrum in terms of the resonances 
associated to resonant states in $\ker U_-$. Indeed the only case which is not obvious is when $\la_0=-n\in-\nn$, and when we want to compute 
$V_{n+k}^j(-n)$ for $k\geq 0$. If $u\in V_{n}^j(-n)$, then $w:=U_-^nu$ is in 
$V_0^j(0)$ which is non trivial only if $j=1$, in which case $w={\rm constant}$. But 
then $||w||_{L^2}^2=(-1)^n\cjg U_-^{2n}w,u\cjd=0$ and we deduce that 
$V_{n}^j(-n)=V_{n-1}^j(-n)$ and $V_{n+k}^j(-n)=V_{n-1}^j(-n)$ for all $k\geq 0$. 

\section{Ruelle resonances for convex co-compact quotients}

In this section, we consider the case of a convex co-compact subgroup $\Gamma\subset G$ 
of isometries of $\hh^{2}$. 

\subsection{Geometry and dynamics of convex co-compact surfaces}\label{cccsurface}

The manifold $M=\Gamma\backslash \hh^{2}$ is a non-compact complete smooth hyperbolic 
manifold with infinite volume but finitely many topological ends. 
Moreover, $M$ can be compactified to the smooth manifold 
$\bbar{M}=\Gamma\backslash (\hh^2\cup \Omega_\Gamma)$, if 
$\Omega_\Gamma\subset \mathbb{S}^1$ is the set of discontinuity of $\Gamma$. As 
in Section \ref{sec:quotients}, we will denote by $\Lambda_\Gamma\subset \mathbb{S}^1$ 
the limit set of $\Gamma$.
In fact, $M$ is conformally compact in the sense of Mazzeo-Melrose \cite{MaMe}: 
there is a smooth boundary defining function $\rho$ such that $\bar{g}:=\rho^2g$ 
extends as a smooth metric on $\bbar{M}$. 

The group $\Gamma$ is a subgroup of ${\rm PSL}_2(\cc)$ and acts on the Riemann sphere 
$\overline \cc := \cc\cup\{\infty\}$ as conformal 
transformations, it preserves the unit disk $\hh^2$ and its complement $\overline \cc\setminus \hh^2$. 
Equivalently, by conjugating by $(z-i)/(z+i)$, $\Gamma$ acts by conformal transformations 
on $\overline \cc$ as a subgroup of ${\rm PSL}_2(\rr)\subset {\rm PSL}_2(\cc)$ and it preserves 
the half-planes $\hh^2_\pm:=\{z\in \cc; \pm{\rm Im}(z)>0\}$. The half-planes are conformally 
equivalent through $z\mapsto \bar{z}$ if we put the opposite orientation on $\hh^2_+$ and $\hh^2_-$. 
In this model the boundary is the compactified real line $\partial \hh_\pm= \overline \rr := \rr\cup\{\infty\}$
and the limit set is a closed subset $\Lambda_\Gamma$ of $\overline \rr$, and 
its complement in $\overline \rr$ is 
still denoted by $\Omega_\Gamma$. Since $\bbar{\gamma(\bar{z})}=\gamma(z)$ for 
each $\gamma\in \Gamma$, the quotients $M_+:=\Gamma\backslash (\hh^2_+\cup \Omega_\Gamma)$ 
and $M_-:=\Gamma\backslash(\hh^2_-\cup \Omega_\Gamma)$ 
are smooth surface with boundaries, equipped with a natural conformal structure 
and $M_+$ is conformally equivalent to $M_-$. The surface 
$\Gamma\backslash (\overline \cc\setminus \Lambda_\Gamma)$ is a compact surface 
diffeomorphic to the gluing 
$M_2:=M_+\cup M_-$ of $M_+$ and $M_-$ along their boundaries, moreover it is 
equipped with a smooth conformal structure which restricts to that of $M_\pm$. 

We denote by $\mc{I}:M_2\to M_2$ the involution 
fixing $\pl\bbar{M}$ and derived from $z\mapsto \bar{z}$ when viewing $\Gamma$ 
as acting in $\overline \cc\setminus \Lambda_\Gamma$.  
The interior of $M_+$ and $M_-$ are isometric if we put the hyperbolic metric 
$|dz|^2/({\rm Im}(z))^2$ on $\hh^2_\pm$, and they are isometric to the hyperbolic 
surface we called $M$ above. The conformal class of $M_\pm$ corresponds to the 
conformal class of $\bar{g}$ on $\bbar{M}$ as defined above. We identify $M_+$ 
with $\bbar{M}$ and define $H_{\pm n}(M)$ as the finite dimensional real vector spaces 
\begin{equation}\label{HnH-n}
\begin{gathered}
 H_n(M):=\{ f|_{M_+}; f\in C^\infty(M_2; \mc{K}^n),\, \bbar{\pl}f=0, \, \mc{I}^*f=\bar{f}\}, \\
 H_{-n}(M):=\{ f|_{M_+}; f\in C^\infty(M_2; \mc{K}^{-n}),\, \pl f=0,\, \mc{I}^*f=\bar{f}\}.
\end{gathered} 
\end{equation}
Note that $f\in H_n(M)$ is equivalent to say that $\bbar{\pl}f=0$ with $\iota_{\pl \bbar{M}}^*f$ real-valued, if $\iota_{\pl \bbar{M}}: \pl \bbar{M}\to \bbar{M}$ is the inclusion map and $\iota_{\pl \bbar{M}}^*f$ is the 
symmetric tensor on $\pl \bbar{M}$ defined by 
\[ \forall (x,v)\in T\pl \bbar{M}, \,\, \iota_{\pl \bbar{M}}^*f(x)(\otimes^mv)=f(\iota_{\pl\bbar{M}}(x))(\otimes^m d\iota_{\pl\bbar{M}}.v)\]
where $T\pl \bbar{M}$ is the real tangent space of $\pl\bbar{M}$. 

The dimension of $H_{\pm n}$ can be computed as follows: 
let 
\[H_n(M_2):=\{f\in C^\infty(M_2; \mc{K}^n),\, \bbar{\pl}f=0\}\] 
which can be viewed as complex and real vector space, with complex
dimension that can be calculated by the Riemann-Roch theorem (\ref{eq:dim_riemann_roch_cpt})
and the fact that $\chi(M_2) = 2\chi(M)$. Now let $A: H_n(M_2)\to H_n(M_2)$ be the map 
$Af=\bbar{\mc{I}^*f}$ satisfying $A^2={\rm Id}$. We have $H_n(M)=\ker (A-{\rm Id})$ as real vector spaces.
The map $f\mapsto if$ is an isomorphism of real vector spaces from $H_n(M)$ 
to $\ker (A+{\rm Id})\subset H_n(M_2)$ and thus we deduce that 
the real dimension of $H_n(M)$ equals the complex dimension of $H_n(M_2)$ and we get 
\begin{equation}\label{dimreHn}
\dim_{\rr}H_n(M)=\left\{\begin{array}{ll}
(2n-1)|\chi(M)| & \textrm{ if }n>1,\\
|\chi(M)|+1 & \textrm{ if }n=1.
\end{array}\right.\end{equation}
for a non-elementary group. For an elementary group, $\chi(M)=0$ and $H_n(M)$ has real dimension equal to $1$ ($M_2$ is a flat torus).
 
By \cite{Gr}, there exists a collar near $\pl \bbar{M}$ and a diffeomorphism 
$\psi: [0,\eps)_r\x \pl \bbar{M}\to \psi([0,\eps)\x \pl \bbar{M})\subset \bbar{M}$ such that 
$\psi^*g=\frac{dr^2+h(r)}{r^2}$ where $r\mapsto h(r)$ is a smooth family of metrics on $\pl\bbar{M}$.
We can choose $\rho=r\circ \psi^{-1}$ as boundary defining function. It is then clear that the 
hypersurfaces $\rho=\rho_0$ are strictly convex if $0<\rho_0<\eps$ is small enough, and therefore 
there exists a geodesically convex compact domain $Q\subset M$ with smooth boundary 
$\pl Q=\psi(\{\rho_0\}\x\pl\bbar{M})$. Each geodesic $(x(t))_{t\in [0,t_0]}$ in $Q$ with 
$x(t_0)\in \pl Q$ can be extended to a geodesic $(x(t))_{t\in[0,\infty)}$ so that 
$x(t)\in M\setminus Q$ for all $t>t_0$ and 
$x(t)\to \pl \bbar{M}$ as $t\to +\infty$. The surface $M$ is of the form 
$M=N\cup(\cup_{i=1}^{n_f} F_i)$ where $N$ is a compact hyperbolic surface with geodesic boundary, 
called the \emph{convex core}, and $F_1,\dots, F_{n_f}$ are $n_f$ ends isometric to funnels 
(se e.g. \cite[Section 2.4]{Bo})
\begin{equation}\label{ends} 
(F_i,g) \simeq [0,\infty)_t\x (\rr/\ell_i \zz)_\theta \textrm{ with metric } dt^2+\cosh(t)^2d\theta^2
\end{equation} 
for some $\ell_i>0$ corresponding to the length of the geodesic of $N$ where $F_i$ is attached. The boundary 
of the compactification is $\pl\bbar{M}=\cup_{i=1}^{n_e} S_i$ where $S_i:=\rr/\ell_i \zz$.
We will now choose the function $\rho$ to be equal to $\rho=2e^{-t}$ in $F_i$ so that 
$\bbar{g}|_{S_i}=d\theta^2$. In the $(\rho,\theta)$ coordinates, the metric in $F_i$ is given by $g=\rho^{-2}(d\rho^2+(1+\rho^2/4)^2d\theta^2)$.
 Using the $(\rho,\theta)$ coordinates, we say that a function $f\in C^\infty(\bbar{M})$ is \emph{even} if in each 
$F_i$, 
\[
\forall k\in\nn_0,\,  \pl_{\rho}^{2k+1}f|_{\rho=0}=0
\] 
and we denote by $C^\infty_{\rm ev}(\bbar{M})$ the space of even functions. Note that
\begin{equation}\label{even}
C^\infty_{\rm ev}(\bbar{M})=\{ f|_{M_+}; f\in C^\infty(M_2), \mc{I}^*f=f\}
\end{equation}
after identifying $\bbar{M}$ and $M_+$, and if $\mc{I}$ is the involution on $M_2$ 
defined above. We refer to \cite{Gu1} for detailed discussions about even 
metrics and functions on asymptotically hyperbolic manifolds.

The incoming (-) and outgoing (+) tails $\Lambda_\pm \subset SM$ of the flow are the sets 
\[
\Lambda_\pm:=\{ (x,v)\in SM; \exists t_0\geq 0,\; \varphi_{\mp t}(x,v)\in Q, \forall t\in [t_0,\infty) \}.
\]
In view of the property of $Q$, the set  $\Lambda_\pm$ is also the set of points 
$(x,v)$ such that $\pi_0(\varphi_{\mp t}(x,v))$ does not tend to $\pl \bbar{M}$ in 
$\bbar{M}$ as $t\to +\infty$. The trapped set $K\subset Q$ is the closed 
flow-invariant set  
\[K:= \Lambda_+\cap \Lambda_-.\]
It is a direct observation that, if $B_\pm:S\hh^{2}\to \mathbb{S}^1$ are the 
endpoint maps defined in Section \ref{hyperbolicspace} and 
$\pi_\Gamma:S\hh^2 \to SM$ the covering map defined in Section \ref{sec:quotients},  then
\begin{equation}\label{Lambada}
\Lambda_\pm=\pi_\Gamma(\til{\Lambda}_\pm),\textrm{ where }\til{\Lambda}_\pm:=B_\mp^{-1}(\Lambda_\Gamma).
\end{equation}
Indeed, a point $(x,v)\in SM$ is in $\Gamma_-$ if and ony if it is trapped in the future, meaning that $\varphi_t(x,v)\not\to \pl\bbar{M}$ as $t\to +\infty$. By lifting to the covering $S\hh^2$ and taking $(\tilde{x},\tilde{v})\in S\hh^2$ so that $\pi_\Gamma(\tilde{x},\tilde{v})=(x,v)$, this means that the semi-geodesic starting at $\tilde{x}$ tangent to $\tilde{v}$ does not have endpoint in $\Omega_\Gamma$, i.e, $B_+(\tilde{x},\tilde{v})\in \Lambda_\Gamma$.
The bundle $T(SM)$ has a continuous (in fact smooth in our case) flow-invariant 
splitting \eqref{anosov} and we set $E_0^*,E_u^*,E_s^*$ the dual splitting 
defined by \eqref{dualspl}.
We define their restriction to the incoming/outgoing tails by 
\begin{equation}\label{E_pm^*}
E_+^* := E_u^*|_{\Lambda_+} ,\quad E_-^*:=E_s^*|_{\Lambda_-}.
\end{equation}

\subsection{Ruelle resonances and generalized resonant states}

To define Ruelle resonances and resonant states, we first need to recall the 
following result from \cite{DyGu}.
\begin{theo}\label{dygu}
If $M=\Gamma\backslash\hh^{2}$ is a convex co-compact hyperbolic surface, then 
the generator $X$ of the geodesic flow on $SM$ has a resolvent 
$R_X(\la):=(-X-\la)^{-1}$ that admits a meromorphic extension from 
$\{\la\in \cc; {\rm Re}(\la)>0\}$ to 
$\cc$ as a family of bounded operators $C_c^\infty(SM)\to \mc{D}'(SM)$. The 
resolvent $R_X(\la)$ has finite rank polar part at each pole $\la_0$ and the 
polar part is of the form 
\[ -\sum_{j=1}^{J(\la_0)} \frac{(-X-\la_0)^{j-1}\Pi^X_{\la_0}}{(\la-\la_0)^{j}}, \quad J(\la_0)\in \nn \]
for some finite rank projector $\Pi^X_{\la_0}$ commuting with $X$. Moreover 
$u\in \mc{D}'(SM)$ is in the range 
of $\Pi^X_{\la_0}$ if and only if $(X+\la_0)^{J(\la_0)}u=0$ with $u$ supported in 
$\Lambda_+$ and ${\rm WF}(u)\subset E_+^*$.  
\end{theo}
We define \emph{Ruelle resonance},  \emph{generalized Ruelle resonant state} and 
\emph{Ruelle resonant state} as respectively a pole $\la_0$ of $R_X(\la)$, an 
element in ${\rm Im}(\Pi^X_{\la_0})$ and an element in 
${\rm Im}(\Pi^X_{\la_0})\cap \ker(-X-\la_0)$.
Define like in \eqref{Resk} and \eqref{Vm} the spaces 
\begin{equation}\label{Resk2}
{\rm Res}^j_X(\la_0):=\{ u\in \mc{D}'(SM) ; {\rm supp}(u)\subset \Lambda_+, \, 
{\rm WF}(u)\subset E_u^*,\,  (-X-\la_0)^ju=0\},
\end{equation}
\begin{equation}\label{Vm2}
V_m^j(\la_0):=\{ u\in {\rm Res}_X^j(\la_0); U_-^{m+1}u=0\}.
\end{equation}
As in the compact case, the operator $(-X-\la_0)$ is nilpotent on the finite dimensional space 
${\rm Res}_X(\la_0):=\cup_{j\geq 1}{\rm Res}^j_X(\la_0)$ and 
$\la_0$ is a Ruelle resonance if and only if ${\rm Res}_X^1(\la_0)\not=0$. The presence of Jordan blocks
for $\la_0$ is equivalent to having ${\rm Res}_X^k(\la_0)\not={\rm Res}_X^1(\la_0) $ for some $k>1$.

We make the important remark for what follows that when $\la_0\in\rr$, $V_m^j(\la_0)$
can be considered both as a real and as a complex vector space
which admits a basis of real-valued distributions, and its dimension is the number of elements of the basis.

\subsection{Quantum resonances and scattering operator} 
Scattering theory on these surfaces has been largely developped by Guillop\'e-Zworski \cite{GuZw,GuZw2} 
and a comprehensive description is given in the book of Borthwick \cite{Bo}. 
Quantum resonances are resonances of the Laplacian $\Delta_M=d^*d$ on 
$M=\Gamma\backslash \hh^{2}$, and these are defined as poles of the meromorphic 
extension of the resolvent of $\Delta_M$. The essential spectrum for $\Delta_M$ 
is $[1/4,\infty)$ and the natural resolvent to consider is $R_{\Delta}(s):=(\Delta_M-s(1-s))^{-1}$ 
which is meromorphic 
in ${\rm Re}(s)>1/2$ as a family of bounded operators on $L^2(M)$, with finitely 
many poles at 
\[
\sigma(\Delta_M):=\{s \in(\demi,1); \ker_{L^2}(\Delta_M-s(1-s))\not=0\}
\] 
corresponding to the $L^2$-eigenvalues in $(0,1/4)$. By usual 
spectral theory, each pole at such $s_0$ is simple and the 
residue is 
\[ {\rm Res}_{s=s_0}(R_{\Delta}(s))=\frac{\Pi^{\Delta}_{s_0}}{(2s_0-1)}\]
where $\Pi^{\Delta}_{s_0}$ the orthogonal projector on $\ker_{L^2}(\Delta_M-s_0(1-s_0))$; 
see \cite[Lemma 4.8]{PaPe} for example. The meromorphic extension of the resolvent 
was proved in \cite{MaMe,GuZw}, we now recall this result:
\begin{theo}\label{mame}
If $M=\Gamma\backslash\hh^{2}$ is a convex co-compact hyperbolic surface, 
then the non-negative Laplacian $\Delta_M$ 
on $M$ has a resolvent $R_{\Delta}(s):=(\Delta_M-s(1-s))^{-1}$ that admits a 
meromorphic extension from 
$\{s\in \cc; \,{\rm Re}(s)>1/2\}$ to 
$\cc$ as a family of bounded operators $C_c^\infty(M)\to C^{\infty}(M)$. 
The resolvent $R_{\Delta}(s)$ has finite rank polar part at the poles and the polar part at 
$s_0\not=1/2$ is of the form 
\[
\sum_{j=1}^{J(s_0)} \frac{(\Delta_M-s_0(1-s_0))^{j-1}\Pi^{\Delta}_{s_0}}{(s(1-s)-s_0(1-s_0))^j},\quad J(s_0)\in \nn, 
\]
for some finite rank operator $\Pi^{\Delta}_{s_0}:=(1-2s_0){\rm Res}_{s_0}(R_{\Delta}(s))$ 
commuting with $\Delta_M$. For $s_0=1/2$,  $R_{\Delta}(s)$ has a pole of order 
at most $1$ at $s_0$, $\Pi^{\Delta}_{s_0}:={\rm Res}_{s_0}(R_{\Delta}(s))$ has 
finite rank and $(\Delta_M-1/4)\Pi^{\Delta}_{s_0}=0$.
The operator 
$R_{\Delta}(s)$ is bounded as a map $\rho^NL^2(M)\to \rho^{-N}L^2(M)$ if $N>|{\rm Re}(s)-1/2|$.
\end{theo}
The main theorem of \cite{MaMe} shows in addition that the Schwartz kernel 
$R_{\Delta}(s;x,x')$ of $R_{\Delta}(s)$ is of the form
\begin{equation}\label{schwartzker}
(\rho(x)\rho(x'))^{-s}R_{\Delta}(s;x,x') \in C^\infty(\bbar{M}\x\bbar{M} \setminus {\rm diag}),
\end{equation}
where ${\rm diag}$ denotes the diagonal of $\bbar{M}$. Moreover, if $f\in C_c^\infty(M)$, 
$u_s:=\rho^{-s}R_{\Delta}(s)f\in C^\infty(\bbar{M})$ is a meromorphic family in
$s\in\cc$, and since the Laplacian $\Delta_M$ in each funnel is given by 
\begin{equation}\label{DeltaM}
\Delta_M=-(\rho\pl_\rho)^2+\frac{1-\rho^2/4}{1+\rho^2/4}\rho\pl_\rho -\frac{\rho^2}{(1+\rho^2/4)^2}\pl^2_\theta,
\end{equation}
a series expansion of $u_s$ in powers of $\rho$ near $\pl\bbar{M}$ directly shows that $u_s\in 
C_{\rm ev}^\infty(\bbar{M})$. Therefore, for each pole $s_0$ of order $j\geq 1$, we get 
\begin{equation}\label{residueexp}
\varphi \in {\rm Ran}(\Pi^{\Delta}_{s_0}) \Rightarrow \varphi\in\bigoplus_{k=0}^{j-1}\rho^{s_0}\log(\rho)^kC^\infty_{\rm ev}(\bbar{M}).
\end{equation}
Indeed, it suffices to consider the residue of $\rho^su_s=R_\Delta(s)f$ at $s=s_0$ using Taylor expansion of $\rho^s$ as $s=s_0$, when $f\in C_c^\infty(M)$ is arbitrary, and to use that $u_s$ is even in $\rho$, and one gets \eqref{residueexp}.

We say that $\varphi$ is a \emph{generalized resonant state} for $s_0$ if 
$\varphi\in {\rm Ran}(\Pi^{\Delta}_{s_0})$, and that  it is a \emph{resonant state} if in 
addition $(\Delta_M-s_0(1-s_0))\varphi=0$. The multiplicity of a quantum resonance $s_0$ is defined to be the rank of $\Pi^{\Delta}_{s_0}$.
We will define the generalized resonant states of order $j\geq 1$ at $s_0$ by
\begin{equation}\label{Resj} 
{\rm Res}_{\Delta}^j(s_0):=\{ \varphi\in {\rm Ran}(\Pi^{\Delta}_{s_0}); (\Delta_M-s_0(1-s_0))^j\varphi=0\}
\end{equation}
with $\Pi^{\Delta}_{s_0}:={\rm Res}_{s=s_0}(R_{\Delta}(s))$.

We need a characterization of generalized resonant states of order $j$ as solutions of 
\[(\Delta_M-s_0(1-s_0))^{j}u=0\]
with very particular asymptotics for $u$ at the boundary $\pl\bbar{M}$ of the 
conformal compactification $\bbar{M}$. For this purpose, we define the Poisson 
operator $\mc{E}_M(s)$ on $M$ 
and the scattering operator $\mc{S}_M(s)$ by following the approach of Graham-Zworski 
\cite{GrZw}; we shall refer to that paper for details. 
By \cite[Proposition 3.5]{GrZw}, there is a meromorphic family of operators 
\[ \mc{E}_M(s): C^\infty(\pl\bbar{M})\to C^\infty(M)\]
in ${\rm Re}(s)\geq 1/2$, with only simple poles at $s\in \sigma(\Delta_M)$ and 
satisfying outside the poles
\[
(\Delta_M-s(1-s))\mc{E}_M(s)f=0
\]
for each $f\in C^\infty(\pl\bbar{M})$ and with the property that this is the 
only solution such that 
there is $F_s,G_s\in C_{\rm ev}^\infty(\bbar{M})$ such that $F_s|_{\pl \bbar{M}}=f$ and
\begin{equation}\label{poissonM}
\begin{gathered}
\mc{E}_M(s)f=\rho^{1-s}F_s+\rho^sG_s \textrm{ if } s\notin \demi+\nn,\\
\mc{E}_M(s)f=\rho^{1/2-k}F_{1/2+k}+\rho^{1/2+k}\log(\rho)G_{1/2+k}  \textrm{ if } s=1/2+k, \,\, k\in\nn. 
\end{gathered}  
\end{equation}
Here $F_s,G_s$ are meromorphic with simple poles at $s\in \sigma(\Delta_M)$
 and $\demi+\nn$. The functions $F_{1/2+k}$, $G_{1/2+k}$ can be expressed in 
 terms of residues of $F_s,G_s$ at $1/2+k$. Notice that in the case 
$M=\hh^2$, $\mc{E}_{\hh^2}(s)=\pi^{-\demi}2^{1-s}\frac{\Gamma(s)}{\Gamma(s-1/2)}\mc{P}_{s-1}$, 
where $\mc{P}_s$ is the Poisson-Helgason transform of Lemma \ref{poissoniso}.
By \cite[Proposition 3.9]{GrZw} the Schwartz kernel of $\mc{E}_M(s)$ is related 
to the Schwartz kernel of $R_{\Delta}(s)$ by 
\begin{equation}\label{schwartzE}
\mc{E}_M(s;x,\theta)= (2s-1)[R_{\Delta}(s;x,x')\rho(x')^{-s}]|_{x'=\theta \in \pl\bbar{M}}.
\end{equation}
This operator thus admits a meromorphic continuation to $s\in \cc$, as $R_{\Delta}(s)$ does.
The scattering operator $\mc{S}_M(s)$ is a meromorphic family of operators acting on $C^\infty(\pl\bbar{M})$ for ${\rm Re}(s)\geq 1/2$, unitary on ${\rm Re}(s)=\demi$, and defined by 
\begin{equation}\label{scatteringdef}
\mc{S}_M(s)f:= 2^{2s-1}\frac{\Gamma(s-\demi)}{\Gamma(\demi-s)}G_s|_{\pl\bbar{M}}.
\end{equation}
This operator is holomorphic outside $\sigma(\Delta_M)$ and is a family of elliptic
pseudo-differential operators of order $2s-1$, which is Fredholm of index $0$ 
as a map $H^{2s-1}(\pl\bbar{M})\to L^2(\pl\bbar{M})$, it extends meromorphically 
to $\cc$ and satisfies the functional equation 
\begin{equation}\label{fcteqS}
\mc{S}_M(s)^{-1}=\mc{S}_M(1-s).
\end{equation} 
By \cite{GrZw}, there are special points, namely $s=\demi+k$ with $k\in \nn$, where $\mc{S}_M(s)$ is a differential operator on $\pl\bbar{M}=\cup_{i=1}^{n_f}S_i$ 
which depends only on the metric $\bar{g}|_{\pl\bbar{M}}$. Moreover we have 
\begin{equation}\label{sppoints}
G_{1/2+k}|_{\pl\bbar{M}}=c_k\, \mc{S}_M(1/2+k)f, \quad c_k\not=0.
\end{equation}
 The computation of the scattering operator
is done by Guillop\'e-Zworski \cite[Appendix]{GuZw} for the hyperbolic cylinder 
$\mc{C}(\ell_i):=\rr_t\x (\rr/\ell_i\zz)_\theta$ with metric $dt^2+\cosh(t)^2d\theta^2$, 
and their computation shows that 
$\ker \mc{S}_{\mc{C}(\ell_i)}(\demi+k)=0$ for all $k\in \nn$, which implies
\begin{equation}\label{SMiso}
\mc{S}_M(\demi+k) : C^\infty(\pl\bbar{M})\to C^\infty(\pl\bbar{M}) \textrm{ is an isomorphism};
\end{equation} 
note that this fact is also proved in \cite[Lemma 8.6]{Bo}. 
Finally, one has a functional equation similar to
\eqref{fcteq},  which follows from the definition of $\mc{E}_M(s)$ and $\mc{S}_M(s)$
\begin{equation}\label{fcteq2} 
\mc{E}_M(s)=2^{1-2s}\frac{\Gamma(\demi-s)}{\Gamma(s-\demi)}\mc{E}_M(1-s)\mc{S}_M(s)
\end{equation}

We have the following result on quantum resonant states.
\begin{prop}\label{caractreson}
Let $M=\Gamma\backslash \hh^{2}$ be a convex co-compact surface. Then the following properties hold:\\
1) There is no quantum resonance at $\demi-\nn$.\\
2) If ${\rm Re}(s_0)\geq 1/2$, the poles of $R_{\Delta}(s)$ at $s_0$ are simple
and $\varphi$ is a resonant state for $s_0$ if and only if $(\Delta_M-s_0(1-s_0))\varphi=0$ with $\varphi\in \rho^{s_0}C^\infty(\bbar{M})$.\\
3) If ${\rm Re}(s_0)<1/2$ with $s_0\notin 1/2-\nn$, a solution $\varphi$ to 
$(\Delta_M-s_0(1-s_0))\varphi=0$ is a resonant state for $s_0$ if and only if
$\varphi\in\rho^{s_0}C^\infty_{\rm ev}(\bbar{M})$, and
a function $\varphi'$ is a generalized resonant state if and only if it is a 
resonant state or there is a resonant state 
$\varphi$ so that $(\Delta_M-s_0(1-s_0))^j\varphi'=\varphi$ for some $j\geq 1$, and 
$\varphi'\in\bigoplus_{k=0}^j\rho^{s_0}\log(\rho)^kC^\infty_{\rm ev}(\bbar{M})$.
\end{prop}
\begin{proof} To prove 1), we use  \cite[Lemma 3.4]{Gu2} which says that if $s_0$ 
is a pole of $R_{\Delta}(s)$ then it is a pole of 
$S_M(s):=2^{1-2s}\frac{\Gamma(1/2-s)}{\Gamma(s-1/2)}\mc{S}_M(s)$. But $S_M(s)$ 
has a pole of order $1$ at $s_0=1/2+k$ with residue $c_k\,\mc{S}_M(1/2+k)$ for 
some $c_k\not=0$, and this operator has no kernel. Therefore, by expanding in 
Laurent series the functional equation $S_M(1-s)S_M(s)={\rm Id}$ at $s_0=1/2-k$, 
we directly see that $S_M(s)$ must be of the form
$S_M(s)=(s-s_0)L(s)$
for some holomorphic family of operator $L(s)$ near $s=s_0$, with $L(s_0)$ 
invertible on $C^\infty(\pl\bbar{M})$. This shows 1). 

Statement 2) is direct to see for ${\rm Re}(s_0)>1/2$: the resonant states are of 
the desired form by \cite[Lemma 4.8]{PaPe} and conversely if
$\varphi\in \rho^{s_0}C^\infty(\bbar{M})\cap \ker(\Delta_M-s_0(1-s_0))$, 
then $\varphi\in \ker_{L^2}(\Delta_M-s_0(1-s_0))={\rm Ran}(\Pi^{\Delta}_{s_0})$. 
For 
$s_0=1/2$, the pole of the resolvent is simple and the resonant states are in 
$\rho^{1/2}C^\infty(\bbar{M})$ by 
 \cite[Lemma 4.9]{PaPe}. The converse part will follow from the proof of 3).
 
Now we prove 3). By \eqref{residueexp}, if $\varphi'$ is a generalized resonant 
state satisfying the equation $(\Delta_M-s_0(1-s_0))^{j+1}
\varphi'=0$, then $\varphi'\in \bigoplus_{k=0}^J \rho^{s_0}\log(\rho)^kC_{\rm ev}^\infty(\bbar{M})$ 
for some $J$. Using the form of $\Delta_M$ in \eqref{DeltaM} and writing the 
formal expansions at $\rho=0$ of that equation, it is direct to see that $J\leq j$.

We next prove the converse.  Let $A_{s_0}:=(\Delta_M-s_0(1-s_0))$ and let 
$\varphi$ be a solution of $A_{s_0}^{j+1}\varphi=0$ with 
$\varphi\in \bigoplus_{\ell=0}^{j}\rho^{s_0}\log(\rho)^{\ell}C_{\rm ev}^\infty(\bbar{M})$ 
with $j\geq 0$. Define $\varphi_\ell:=A_{s_0}^{j-\ell}\varphi$ for all $\ell\in 0,\ldots,j$.
It is easy to check that 
$\varphi_\ell\in \bigoplus_{k=0}^\ell\rho^{s_0}\log(\rho)^{k}C_{\rm ev}^\infty(\bbar{M})$. We first 
construct a holomorphic family $\phi(s)=\rho^s F(s)$ with  $F(s)\in C_{\rm ev}^\infty(\bbar{M})$ 
such that 
\begin{equation}\label{defqs} 
q(s):=(\Delta_M-s(1-s))\phi(s)\in\rho^{s+2}C_{\rm ev}^\infty(\bbar{M}), \quad |q(s)|\leq C|s-s_0|^{j+1}\rho^{{\rm Re}(s)+2}.
\end{equation}
To construct $F(s)$, we will set $F(s):=\sum_{k=0}^j F_k(s-s_0)^k$ for some 
$F_k\in C_{\rm ev}^\infty(\bbar{M})$ well chosen. Taylor expanding at $s=s_0$
\begin{equation}\label{phi(s)}
\phi(s)=\rho^s F(s)=\sum_{k=0}^j(s-s_0)^k\phi_k+\mc{O}((s-s_0)^{j+1})
\end{equation}
with $\phi_k:=\rho^{s_0}\sum_{\ell=0}^k \frac{1}{\ell!}\log(\rho)^\ell F_{k-\ell}$, the equation 
\[ (\Delta_M-s(1-s))\phi(s)=\mc{O}((s-s_0)^{j+1})\]
reduces to 
\begin{equation}\label{inductionA}
 A_{s_0}\phi_k+(2s_0-1)\phi_{k-1}+\phi_{k-2}=0
 \end{equation} 
for all $k\leq j$, with the convention $\phi_{-1}=\phi_{-2}=0$.
The equation \eqref{inductionA} can be solved by 
choosing $\phi_k$  to be a linear combination of the form 
\[ \phi_0=\varphi_0, \textrm{ and } \phi_k= (1-2s_0)^k\varphi_k +\sum_{\ell=0}^{k-1}c_\ell(s_0)\varphi_{\ell} \,\, \textrm{ for }k\in[1,j]\]
for some polynomials $c_\ell(s_0)$ in $s_0$,
and we also note that $\phi_k\in \bigoplus_{\ell=0}^k\rho^{s_0}\log(\rho)^{\ell}C_{\rm ev}^\infty(\bbar{M})$.
We can then define for $k\leq j$
\[  F_k:= \rho^{-s_0}\sum_{\ell=0}^k \frac{1}{\ell!}(-\log \rho)^\ell \phi_{k-\ell}.
\]
We need to check that $F_k\in C_{\rm ev}^\infty(\bbar{M})$. To do this, we will show that 
\begin{equation}\label{stufftoshow}
\rho^{-s_0}A_{s_0}(\rho^{s_0}F_k)\in \rho^2C_{\rm ev}^\infty(\bbar{M}).
\end{equation}
Indeed, since we know that $F_k\in \bigoplus_{\ell=0}^{2k} \log(\rho)^\ell C_{\rm ev}^\infty(\bbar{M})$, it is an easy exercise to check that \eqref{stufftoshow} implies that $F_k\in C_{\rm ev}^\infty(\bbar{M})$ by using the expression \eqref{DeltaM} of $\Delta_M$ near $\pl\bbar{M}$. 
We already know that $F_0\in C_{\rm ev}^\infty(\bbar{M})$ and that \eqref{DeltaM} is true for $k=0$.
Now to prove \eqref{stufftoshow}, we write 
\[ A_{s_0}(\rho^{s_0}F_k)=\sum_{\ell=0}^k \frac{(-1)^\ell}{\ell!}\Big(A_{s_0}(\phi_{k-\ell})(\log \rho)^\ell+
\Delta_M((\log \rho)^\ell)\phi_{k-\ell}-2\ell (\log\rho)^{\ell-1} N\phi_{k-\ell}\Big)\]
where $N:=\nabla(\log \rho)$ denotes the gradient of $\log\rho$ with respect to $g$.
By \eqref{inductionA}, we get 
\[ \sum_{\ell=0}^k \frac{(-1)^\ell}{\ell!}A_{s_0}(\phi_{k-\ell})(\log \rho)^\ell=\rho^{s_0}(-F_{k-2}+(1-2s_0)F_{k-1}).
\]
Next we have
\[\begin{gathered} 
\sum_{\ell=1}^k \frac{(-1)^\ell}{\ell!}\ell (\log\rho)^{\ell-1} N\phi_{k-\ell}=\\
\sum_{\ell=1}^k
\frac{(-1)^\ell}{(\ell-1)!}N((\log\rho)^{\ell-1}\phi_{k-\ell})-
|N|^2_g\, \sum_{\ell=2}^k
\frac{(-1)^\ell}{(\ell-2)!} (\log\rho)^{\ell-2}\phi_{k-\ell}=\\
-\rho^{s_0}((N+s_0|N|_g^2)F_{k-1}+ |N|_g^2F_{k-2}).
\end{gathered}\]
and
\[\begin{gathered}
\sum_{\ell=0}^k \frac{(-1)^\ell}{\ell!}\Delta_M((\log \rho)^\ell)\phi_{k-\ell}=\\
-|N|^2_g\sum_{\ell=2}^k \frac{(-1)^\ell}{(\ell-2)!}
(\log \rho)^{\ell-2}\phi_{k-\ell}+\sum_{\ell=1}^k \frac{(-1)^\ell}{(\ell-1)!}
(\log \rho)^{\ell-1}\phi_{k-\ell}\Delta_M(\log \rho)=\\
-|N|^2_g\rho^{s_0}F_{k-2}-\Delta_M(\log \rho)\rho^{s_0}F_{k-1}.
\end{gathered}
\]
Consequently we get 
\begin{equation}\label{Acommut} 
\rho^{-s_0}A_{s_0}(\rho^{s_0}F_k)=(|N|^2_g-1)F_{k-2}+(2N+1-\Delta_M(\log \rho)+2s_0(|N|^2_g-1))F_{k-1}.
\end{equation}
By using \eqref{DeltaM}, a direct computation gives that $\Delta_M(\log \rho)-1\in \rho^2C_{\rm ev}^\infty(\bbar{M})$ and we also 
have $|N|^2_g=1$ near $\pl \bbar{M}$, so $|N|^2_g\in C_{\rm ev}^\infty(\bbar{M})$. Moreover $N$  maps $C_{\rm ev}^\infty(\bbar{M})$ to $\rho^2C_{\rm ev}^\infty(\bbar{M})$.
An induction in $k$ with \eqref{Acommut}  then shows \eqref{stufftoshow}. We directly get that the function $q(s)$ defined by \eqref{defqs} is a holomorphic family in $\rho^{s+2} C_{\rm ev}^\infty(\bbar{M})$, and by Taylor expanding the term $(\Delta_M-s(1-s))\phi(s)$ at $s=s_0$,
we also have (by construction)
\begin{equation}\label{remainderq}
|\rho^{-s}q(s)|=\mc{O}(|s-s_0|^{j+1}\rho^2).
\end{equation}
Using Green's formula in the region $\rho\geq \eps$ for some small $\eps>0$, 
 we see that for $z\in M$ fixed and $s$ near $s_0$
\[ \begin{split}
(R_\Delta (1-s)q(s))(z)=& \phi(s;z) -\int_{\rho=\eps}\rho\pl_{\rho}R_{\Delta}(1-s;z, \rho,\theta)\phi(s;\rho,\theta) 
\frac{d\theta}{\rho}\\
 &+ \int_{\rho=\eps}R_{\Delta}(1-s;z, \rho,\theta)\rho\pl_{\rho}\phi(s;\rho,\theta) 
\frac{d\theta}{\rho}.
\end{split}\]
Now $\rho\pl_\rho \phi(s;\rho,\theta)=s\rho^{s}F(s;0,\theta)+\mc{O}(\rho^{{\rm Re}(s)+1})$ and, 
 using \eqref{schwartzE}, 
 \[\rho\pl_\rho R_{\Delta}(1-s;z, \rho,\theta)=\frac{1-s}{1-2s}\mc{E}_M(1-s;z,\theta)+\mc{O}(\rho^{2-{\rm Re}(s)}).\]
Thus we obtain
\[R_\Delta (1-s)q(s)=\phi(s)-\mc{E}_M(1-s)F(s)|_{\pl \bbar{M}}.\]

By \eqref{remainderq}, we also have $|R_\Delta (1-s)q(s)|=\mc{O}(|s-s_0|^{j+1})$ uniformly on compact sets of $M$, and therefore 
\[ \phi(s)=\mc{E}_M(1-s)F(s)|_{\pl \bbar{M}}+\mc{O}(|s-s_0|^{j+1})\]
uniformly on compact sets.
We define $f(s):=F(s)|_{\pl \bbar{M}}$ and differentiate for $\ell\leq j$
\[
\pl_s^\ell(\phi(s))|_{s=s_0}=\pl_{s}^{\ell}(\mc{E}_M(1-s)f(s))|_{s=s_0}=\rho^{s_0}\sum_{i=0}^\ell \log(\rho)^i H^{\ell}_i+\rho^{1-s_0}
\sum_{i=0}^\ell \log(\rho)^i G_i^\ell \]
where $H_i^\ell,G_i^\ell\in C^\infty_{\rm ev}(\bbar{M})$ and 
$G_{0}^\ell|_{\pl\bbar{M}}=\pl_s^{\ell}(\mc{S}_M(1-s)f(s))|_{s=s_0}$. But from \eqref{phi(s)} and the fact that 
$\phi_k\in \bigoplus_{\ell=0}^k\rho^{s_0}\log(\rho)^{\ell}C_{\rm ev}^\infty(\bbar{M})$, we see that 
$\pl_s^{\ell}(\mc{S}_M(1-s)f(s))|_{s=s_0}=0$ for all $\ell\leq j$ and thus 
 $\mc{S}_M(1-s)f(s)=\mc{O}(|s-s_0|^{j+1})$.  Therefore $\mc{S}_M(1-s)f(s)=(s-s_0)^{j+1}r(s)$ for some holomorphic family $r(s)\in C^\infty(\pl\bbar{M})$. 
We write $\mc{E}_M(s)=\sum_{\ell=1}^N(s-s_0)^{-\ell}Q_\ell +H(s)$ for some holomorphic operator family 
 $H(s)$ near $s_0$ and some operators $Q_\ell:C^\infty(\pl\bbar{M})\to C^\infty(M)$.
By using \eqref{fcteq2}, we get
\[\mc{E}_M(1-s)f(s)=(s-s_0)^{j+1} \mc{E}_M(s)r(s)=\sum_{\ell=1}^N  (s-s_0)^{-\ell+j+1}Q_{\ell}\,r(s)+
(s-s_0)^{j+1}H(s)r(s)\]
which implies that $j+1\leq N$ and $\pl_s^{k}(\phi(s))|_{s=s_0}\in \bigoplus_{\ell=1}^N{\rm Ran}(Q_\ell)$ for each $k\leq j$.  By \eqref{schwartzE}, 
we have $\bigoplus_{\ell=1}^N{\rm Ran}(Q_\ell)\subset {\rm Ran}(\Pi^{\Delta}_{s_0})$ since 
the singular part of the Laurent expansion of $R_{\Delta}(s)$ is a finite rank operator with range 
${\rm Ran}(\Pi^{\Delta}_{s_0})$. Using again \eqref{phi(s)}, this shows that $\phi_k$, and therefore $\varphi_k$, are generalized resonant states for $k\leq j$. This completes the proof.
\end{proof}

\begin{rem}
The proof of 2) and 3) in Proposition \ref{caractreson} also applies in the more general setting of even asymptotically hyperbolic manifolds in the sense of \cite{Gu1}, in any dimension $n+1$, 
by replacing $(\Delta_M-s_0(1-s_0))$ by $(\Delta_M-s_0(n-s_0))$ and $s_0\notin 1/2-\nn$ by $s_0\notin n/2-\nn$. 
\end{rem}

\begin{lemm}\label{technical}
For any $\varphi\in \rho^{s_0}C^\infty(\bbar{M})$ satisfying $(\Delta_M-s_0(1-s_0))\varphi=0$ 
with $s_0\notin -\nn_0$, there exists a distribution $\omega\in \mc{D}'(\mathbb{S}^1)$ supported in $\Lambda_\Gamma$ such that $\pi_\Gamma^*\varphi=\mc{P}_{s_0-1}(\omega)$ and for each $\gamma\in \Gamma$,
$\gamma^*\omega=N_\gamma^{-s_0+1}\omega$. 
\end{lemm}
\begin{proof} Let $\til{\varphi}=\pi_\Gamma^*\varphi$ be the lift of $\varphi$ to $\hh^{2}$. Then we have 
$(\Delta_{\hh^2}-s_0(1-s_0))\til{\varphi}=0$ on $\hh^{2}$ and we claim that 
$\til{\varphi}$ is tempered on $\hh^{2}$. Indeed, if $T>0$ and $0$ denotes the 
center the unit ball in $\rr^{2}$ (representing $\hh^{2}$), 
consider $m(T):=\sup_{d_{\hh^2}(x,0)\leq T}|\til{\varphi}(x)|$ where 
$d_{\hh^2}(\cdot,\cdot)$ denotes the hyperbolic distance. For 
$\eps>0$ small enough $M_\eps:= \{x\in M; \rho(x)\geq \eps\}$ is a geodesically convex set 
and it is easy to see that there exists $C>0$ so that for all $T>0$ 
each point $x\in \hh^{2}$ with $d_{\hh^2}(x,0)\leq T$ projects by the covering map $\pi_\Gamma$ 
to the region $M_\eps$ for $\eps=Ce^{-T}$. Then 
$m(T)\leq \sup_{x\in M_\epsilon}(|\varphi(x)|)\leq C_{s_0}e^{\max(-{\rm Re}(s_0),0)T}$ for some constant 
$C_{s_0}$ depending on ${\rm Re}(s_0)$. Here the last inequality follows from 
$\varphi \in \rho^{s_0}C^\infty(\bbar M)$
and this estimate shows that $\til{\varphi}$ is tempered on $\hh^{2}$. By the surjectivity of the 
Poisson-Helgason transform, there exists a distribution $\omega\in \mc{D}'(\mathbb{S}^1)$ so that 
$\til{\varphi}=\mc{P}_{s_0-1}(\omega)$. By Lemma \ref{poissoniso} and the discussion that follows, 
for any $\chi\in C^\infty(\mathbb{S}^1)$
one has for $t\in(0,\eps)$ with $\eps>0$ small
\begin{equation}\label{weakasymptotic} 
\int_{\mathbb{S}^1} \mc{P}_{s_0-1}(\omega)(\tfrac{2-t}{2+t}\nu)\chi(\nu)d\nu=
\left \{\begin{array}{ll}
t^{1-s_0}F_-(t)+t^{s_0}F_+(t) & \textrm{if } s_0\notin 1/2+\zz\\ 
t^{1-s_0}F_-(t)+t^{s_0}\log(t)F_+(t) & \textrm{if } s_0\in 1/2+\nn\\
t^{1-s_0}\log(t)F_-(t)+t^{s_0}F_+(t) & \textrm{if } s_0\in 1/2-\nn_0
\end{array}\right.
\end{equation}
for some $F_\pm\in C^\infty([0,\eps))$ and $F_-(0)=C(s_0)\cjg \omega,\chi\cjd$ where $C(s_0)\not=0$
because $s_0\notin -\nn_0$.  On the other hand, since $\pi_\Gamma^*\rho$ is a 
boundary defining function of $\Omega_\Gamma$ in $\hh^{2}\cup \Omega_{\Gamma}$, 
in a small neighborhood $V_p$ of any point $p\in \Omega_\Gamma$ in $\hh^{2}\cup \Omega_\Gamma$, 
the function $\til{\varphi}$ has an asymptotic expansion as $t\to 0$
\[
\til{\varphi}(\tfrac{2-t}{2+t}\nu)\sim \sum_{k=0}^\infty t^{s_0+k}\alpha_k(\nu)
\]
for some $\alpha_k\in C^\infty(V_p)$, therefore if $\chi\in C_c^\infty(\Omega_\Gamma)$ 
is supported in $V_p$, we have 
\[
\int_{\mathbb{S}^1} \til{\varphi}(\tfrac{2-t}{2+t}\nu)\chi(\nu)d\nu\sim \sum_{k=0}^\infty t^{s_0+k}\cjg \alpha_k,\chi\cjd
\]
and thus from \eqref{weakasymptotic} we deduce that $\cjg \omega,\chi\cjd=0$. 
This shows that $\omega$ is supported in $\Lambda_\Gamma$. Let $\gamma\in \Gamma$, we have 
\[
\mc{P}_{s_0-1}(N_\gamma^{s_0-1}\gamma^*\omega) = \gamma^* \mc{P}_{s_0-1}(\omega)
\]
which is also equal to $\mc{P}_{s_0-1}(\omega)$ since $\til{\varphi}$ is $\Gamma$-automorphic. 
By the injectivity of the Poisson-Helgason transform \cite[Corollary 6.9]{DFG}, we thus 
deduce that $\gamma^*\omega=N_\gamma^{-s_0+1}\omega$. 
\end{proof}
 
We show the following 
\begin{theo}\label{classicquantic}
Let $M=\Gamma\backslash \hh^{2}$ be a smooth oriented convex co-compact hyperbolic 
surface and let $SM$ be its unit tangent bundle.\\ 
1) For each $\la_0\in \cc\setminus (-\demi-\demi \nn_0)$  the pushforward map ${\pi_0}_*: \mc{D}'(SM)\to \mc{D}'(M)$ restricts to a linear isomorphism of complex vector spaces for each $j\geq 1$
\begin{equation}\label{pi*iso}
{\pi_0}_* : V_0^j(\la_0) \to {\rm Res}_{\Delta}^j(\la_0+1)
\end{equation}
where $\Delta_M$ is the Laplacian on $M$ acting on functions.\\
2) For each  $\la_0=-\demi-k$ with $k\in\nn$, $V_0^j(\la_0)=0$ and ${\rm Res}^j_{\Delta_M}(\la_0+1)=0$ 
for all $j\in\nn$.\\
3) For $\la_0=-\demi$, there are no Jordan blocks, i.e. $V_0^j(-\demi)=0$ for $j>1$, 
and the map 
\begin{equation}\label{pi*iso2} 
{\pi_0}_* : V_0^1(-\demi)\to {\rm Res}_{\Delta}^1(1/2)
\end{equation}
is a linear isomorphism of complex vector spaces.\\
4) For $\la_0=-n\in -\nn$, if $\Gamma$ is non-elementary, there are no Jordan blocks, 
i.e. $V_0^j(-n)=0$ if $j>1$, and the following map is an isomorphism of real vector spaces
\begin{equation}\label{pi*iso3}
 i^{n+1}{\pi_n}_* : V_0^1(-n)\to H_{n}(M)
 \end{equation}
where $H_n(M)$ is defined by \eqref{HnH-n}. 
\end{theo}
\begin{proof} \textbf{Case $\la_0\in \cc\setminus (-\demi-\demi \nn_0)$, injectivity.}
The map $X+\la_0$ is a linear nilpotent map preserving 
the finite dimensional vector space $V_0(\la_0):=\bigoplus_{j\geq 1}V_0^j(\la_0)$ 
of generalized Ruelle resonant states in $\ker U_-$. 
Thus there is a decomposition into Jordan blocks for $X$ on $V_0(\la_0)$: 
for each Jordan block of size 
$j$, one has a $u^{(0)}\in V_0^1(\la_0)$ and some 
$u^{(k)}\in V_0^{k+1}(\la_0)$ for $k\in [1, j]$ satisfying
$(X+\la_0)u^{(k)}=u^{(k-1)}$.  We lift each 
$u^{(k)}\in\mathcal D'(SM)$ to $S\hh^{2}$ and get $\til{u}^{(k)}\in \mc{D}'(S\hh^{2})$ 
so that $\gamma^*\til{u}^{(k)}=\til{u}^{(k)}$ for all $\gamma\in \Gamma$, and $\til{u}^{(k)}$ is 
supported in $\til{\Lambda}_+$, where $\til{\Lambda}_+$ is defined by \eqref{Lambada}.
Define for $k\geq 0$
\begin{equation}\label{varphiik} 
\varphi_{k}:={\pi_0}_*u^{(k)}, \quad \til{\varphi}_{k}:={\pi_0}_*\til{u}^{(k)}.
\end{equation} 
From $u^{(0)} \in V^1_0(\lambda_0)$ we have that $(X+\la_0)u^{(0)}=0$, $U_-u^{(0)}=0$, 
and $u^{(0)}$ is supported in $\Lambda_+$. The same equations hold for $\til{u}^{(0)}$ on $\hh^2$. 
Take the distribution $v^{(0)}:=\Phi_-^{-\la_0}\til{u}^{(0)}$ satisfying $Xv^{(0)}=0$. 
Then there exists $\omega^{(0)}\in \mc{D}'(\mathbb{S}^1)$ 
so that $\mc{Q}_-\omega^{(0)}=v^{(0)}$ by  (\ref{Qpmiso}). Since 
${\rm supp}(\til{u}^{(0)})\subset \til{\Lambda}_+$ and $\mc{Q}_-\omega^{(0)}=B_-^*\omega^{(0)}$, 
using (\ref{Lambada}) we directly get that $\supp(\omega^{(0)})\subset \Lambda_{\Gamma}$.
Using $\gamma^* \tilde u^ {(0)}=\tilde u^{(0)}$ together with (\ref{changeofPhi}),
we have that for any $\gamma\in \Gamma$
\[
\gamma^*\omega^{(0)}=N_\gamma^{-\la_0}\omega^{(0)} \textrm{ with }N_\gamma(\nu)=|d\gamma(\nu)|^{-1}.
\]
We get that $\til{\varphi}_{0}=\mc{P}_{\la_0}(\omega^{(0)})={\pi_0}_*\til{u}^{(0)}$ satisfies
\[
(\Delta_{\hh^2}+\la_0(1+\la_0))\til{\varphi}_{0}=0 
\]
on $\hh^{2}$. Now by Lemma \ref{poissoniso}, $\til{\varphi}_0\not=0$ if 
$\la_0\notin -\nn_0$ and $u^{(0)}\not=0$, thus $\varphi_0$ is a non-zero solution on $M$ of
\[
(\Delta_M+\la_0(1+\la_0))\varphi_{0}=0. 
\]
To prove that $s_0=\la_0+1$ is a quantum resonance, we will use Proposition \ref{caractreson}, 
and for that it is sufficient to prove that $\varphi_{0}\in \rho^{s_0}C_{\rm ev}^\infty(\bbar{M})$, 
and in fact 
$\varphi_{0}\in \rho^{s_0}C^\infty(\bbar{M})$ is sufficient since we assumed $\la_0\notin -\nn$. 
Take a point $p\in \pl \bbar{M}$, and consider $\til{p}\in \Omega_{\Gamma}$ a lift of $p$ to 
$\hh^{2}\cup\Omega_{\Gamma}$. To prove the desired statement, we take the boundary 
defining function $\rho_0(x):=2(1-|x|)/(1+|x|)$ in the closed ball $\bbar{\hh}^{2}$ and we will show that 
$\rho_0^{-s_0}\til{\varphi}_{0}$ is a smooth function near the boundary of $\hh^{2}\cup \Omega_\Gamma$.
Note that 
$\rho_0(x)^{-1}P(x,\nu)$ is smooth outside the subset 
$\{(x,\nu) \in \bbar{\hh}^2\x \mathbb{S}^1; x\not=\nu\}$ and since
$\omega^{(0)}$ is supported in $\Lambda_{\Gamma}$, we deduce directly that 
$\rho_0^{-s_0}\mc{P}_{\la_0}(\omega^{(0)})$ is smooth in a neighbourhood of 
$\tilde p$ in $\hh^{2}\cup \Omega_\Gamma$. We have proved that $\varphi_0$ is a quantum resonant 
state which in addition has asymptotic behaviour given in terms of the distribution 
$\omega^{(0)}$: there is an explicit constant $C(s_0)\not=0$ so that
\begin{equation}\label{boundaryvalueres} 
\pi_{\Gamma}^*([\rho^{-s_0}\varphi_0]|_{\pl \bbar{M}})=C(s_0)\eta^{s_0}\mc{S}(s_0)(\omega^{(0)})
\end{equation}
where $\eta\in C^\infty(\Omega_\Gamma)$ is defined by 
$\pi_\Gamma^*(\rho) \eta=\rho_0+\mc{O}(\rho_0^2)$ near $\Omega_{\Gamma}$.
For the generalized resonant states, we proceed like in the compact case: define for $k\leq j$ 
\[
v^{(k)}:=\Phi_-^{-\la_0}\sum_{\ell=0}^{k} \frac{(\log \Phi_-)^{k-\ell}}{(k-\ell) !}\til{u}^{(\ell)}
\] 
which satisfies $Xv^{(k)}=0$ and $U_-v^{(k)}=0$. Thus there is a distribution 
$\omega^{(k)}$ supported in $\Lambda_{\Gamma}$ such that $\mc{Q}_-\omega^{(k)}=v^{(k)}$, 
and using that $\gamma^*\til{u}^{(k)}=\til{u}^{(k)}$ for all $\gamma\in \Gamma$, we get
\begin{equation}\label{covariance2}
\gamma^*v^{(k)}=N_\gamma^{-\la_0}\sum_{\ell=0}^k\frac{(\log N_\gamma)^\ell}{\ell !}v^{(k-\ell)}
, \quad 
\gamma^*\omega^{(k)}=N_\gamma^{-\la_0}\sum_{\ell=0}^k\frac{(\log N_\gamma)^\ell}{\ell !}\omega^{(k-\ell)}.
\end{equation}
Writing $\til{u}^{(k)}$ in terms of the $v^{(\ell)}$ and using \eqref{deriveepoisson}, we have 
\[
\til{u}^{(k)}= \Phi_-^{\la_0} \sum_{\ell=0}^k \frac{(-\log \Phi_-)^{\ell}v^{(k-\ell)}}{\ell !},\quad 
 \til{\varphi}_{k}=\sum_{\ell=0}^k \frac{ (-1)^{\ell}\pl_{\la}^{\ell}\mc{P}_{\la_0}(\omega^{(k-\ell)})}{\ell !}.
\]
By \eqref{deriveepoisson} or the proof of Lemma \ref{u0=0} we deduce that
\begin{equation} \label{tilvarphi}
(\Delta_{\hh^2}+\la_0(1+\la_0))\til{\varphi}_{k}=(1+2\la_0)\til{\varphi}_{k-1}-\til{\varphi}_{k-2}.
\end{equation} 
(with the convention $\til{\varphi}_{i}=0$ for $i<0$).
This implies the following identities on $M$
\[
(\Delta_{\hh^2}+\la_0(1+\la_0))\sum_{\ell=1}^k\frac{\varphi_{\ell}}{(1+2\la_0)^{k+1-\ell}}=\varphi_{k-1}.
\]
Using that $\pl_{\la}^k((P(x,\nu))^{\la+1})|_{\la_0}$ is of the form 
$\rho_0^{\la_0+1} \sum_{\ell=0}^k (\log(\rho_0(x)))^kH_k(x,\nu)$ for some functions $H_k$ 
smooth in $\{(x,\nu) \in \bbar{\hh}^{2}\x \mathbb{S}^1; x\not=\nu\}$, we deduce 
like we did for $\varphi_{0}$ that 
$\varphi_{k}\in \bigoplus_{\ell=0}^k \rho^{s_0}\log(\rho)^{\ell}C^\infty(\bbar{M})$ 
with $s_0=\la_0+1$. Then, by Proposition \ref{caractreson}, the function $\varphi_{k}$ 
is a quantum generalized resonant state in ${\rm Res}_{\Delta}^{k+1}(s_0)$ for each 
$k=0,\dots,j$, and ${\pi_0}_*:V_0^k(\la_0)\to {\rm Res}_{\Delta}^k(\la_0+1)$ is injective.

 \textbf{Case $\la_0\in \cc\setminus (-\demi-\demi \nn_0)$, surjectivity.}
Next we will show that this map is also surjective: let $s_0\notin \tfrac{1}{2}-\tfrac{1}{2}\nn_0$
be a pole of $R_{\Delta}(s)$ and denote by $\Pi^{\Delta}_{s_0}$ its residue. Recall that 
$F_{s_0}:={\rm Ran}(\Pi^{\Delta}_{s_0})$ is finite dimensional and that 
$A_{s_0}:=(\Delta_M -s_0(1-s_0))|_{F_{s_0}}$ is a linear nilpotent operator preserving this
finite dimensional space. Thus there is a decomposition into Jordan blocks 
for $A_{s_0}$ on $F_{s_0}$: for each Jordan block,
there is a $\phi_{0}\in F_{s_0}$ so that $A_{s_0}\phi_{0}=0$ and some $\phi_{k}\in 
F_{\la_0}$ for 
$k\leq j$ so that $A_{s_0}\phi_{k}=\phi_{k-1}$.  Note that by definition (\ref{Resj}) 
we have $\phi_k\in \mathrm{Res}^{k+1}_{\Delta_M}(s_0)$. We lift 
$\phi_{k}$ to $\hh^{2}$, we get $\til{\phi}_{k}:=\pi_\Gamma^*\phi_{k}\in C^{\infty}(\hh^{2})$.
Using $A_{s_0}\phi_{k}=\phi_{k-1}$ for each $k\geq 1$, we see that there exist $\til{\varphi}_{k}\in C^\infty(\hh^{2})$ so that $\til{\varphi}_{0}=\til{\phi}_{0}$ and 
satisfying \eqref{tilvarphi}: $\til{\varphi}_{k}$ are linear combinations of $(\til{\phi}_{\ell})_{\ell=0,\dots,k}$ 
and are thus $\Gamma$-invariant and descend to some function $\varphi_k$.
We will then show that there exist $\omega^{(k)}\in \mc{D}'(\mathbb{S}^1)$ 
supported in $\Lambda_\Gamma$ so that for all $\gamma\in \Gamma$ and $k\leq j$
\begin{equation}\label{covarianceomegaik}
 \til{\varphi}_{k}=\sum_{\ell=0}^k \frac{ (-1)^{\ell}\pl_{\la}^{\ell}\mc{P}_{\la_0}(\omega^{(k-\ell)})}{\ell !}, \quad \gamma^*\omega^{(k)}=N_\gamma^{-\la_0}\sum_{\ell=0}^k\frac{(\log N_\gamma)^\ell}{\ell !}\omega^{(k-\ell)}\end{equation}
where $\la_0=s_0-1$.
We prove this by induction on $k$. For $k=0$, this is a consequence of Lemma \ref{technical}. 
Suppose that \eqref{covarianceomegaik} is 
satisfied with $k$ replaced by $m$ for all $m\leq k$, and we will show that the same hold at 
order $k+1$. We set 
\[
\psi_{k+1}:=\til{\varphi}_{k+1}+\sum_{\ell=0}^{k}\frac{(-1)^\ell\pl_{\la}^{\ell+1}\mc{P}_{\la_0}(\omega^{(k-\ell)})}{(\ell+1)!}
\]
and using \eqref{tilvarphi} and  \eqref{deriveepoisson}
\[ \begin{split}
A_{s_0}\psi_{k+1} = & -(1+2\la_0)
\sum_{\ell=0}^{k}\frac{(-1)^\ell\pl_{\la}^{\ell}\mc{P}_{\la_0}(\omega^{(k-\ell)})}{\ell!}- \sum_{\ell=1}^{k}\frac{(-1)^\ell\pl_{\la}^{\ell-1}\mc{P}_{\la_0}(\omega^{(k-\ell)})}{(\ell-1)!}
 \\ 
& +(1+2\la_0)\til{\varphi}_{k}-\til{\varphi}_{k-1}=0
\end{split}\]
where the last equality follows by using the first equation of \eqref{covarianceomegaik} at order $k-1$ and $k$. 
The surjectivity of the Poisson-Helgason transform in Lemma \ref{poissoniso} implies that 
there exists $\omega^{(k+1)}\in \mc{D}'(\mathbb{S}^1)$ such that 
$\psi_{k+1}=\mc{P}_{\la_0}(\omega^{(k+1)})$.
Now by definition of $\psi_{k+1}$, we have near each point $p\in \Omega_\Gamma$ 
that $\psi_{k+1} \in \bigoplus_{\ell=0}^{k+1} (\log\rho_0)^\ell \rho^{s_0}C^\infty(\hh^2\cup\Omega_\Gamma)$.
Since $(\Delta_{\hh^2} -s_0(1-s_0))\psi_{k+1} = 0$, in fact one has
$\psi_{k+1} \in \rho^{s_0}C^\infty(\hh^2\cup\Omega_\Gamma)$. Then the same 
arguments as in the proof of Lemma \ref{technical} imply that $\omega^{(k+1)}$ 
is supported in $\Lambda_{\Gamma}$.
This shows the first equation of \eqref{covarianceomegaik} at order $k+1$. Using 
that $\gamma^*\til{\varphi}_{k+1}=\til{\varphi}_{k+1}$ for all $\gamma\in \Gamma$, 
the induction assumption \eqref{covarianceomegaik} implies that 
\[
\begin{split} 
\gamma^*\psi_{k+1}-\psi_{k+1} =& \mc{P}_{\la_0}\Big(\sum_{\ell=1}^{k+1}\frac{(\log N_\gamma)^\ell 
\omega^{(k+1-\ell)}}{\ell!}\Big) \end{split}
\]
but this is also equal to $\mc{P}_{\la_0}(N_\gamma^{\la_0}\gamma^*\omega^{(k+1)}-\omega^{(k+1)})$, which by injectivity of the Poisson-Helgason transform implies 
\[
N_\gamma^{\la_0}\gamma^*\omega^{(k+1)}-\omega^{(k+1)}= \sum_{\ell=1}^{k+1}\frac{(\log N_\gamma)^\ell \omega^{(k+1-\ell)}}{\ell!}
\]
which is exactly \eqref{covarianceomegaik} for $\omega^{(k+1)}$. Now we define the distributions on $S\hh^2$
\[
v^{(k)}:=\mc{Q}_-\omega^{(k)}, \quad  
\til{u}^{(k)}:=\Phi_-^{\la_0} \sum_{\ell=0}^k \frac{(-\log \Phi_-)^{\ell}v^{(k-\ell)}}{\ell !}.
\]
By construction we have, for each  $k\geq 0$, $(X+\la_0)\til{u}^{(k)}=\til{u}^{(k-1)}$ and 
$U_-\til{u}^{(k)}=0$ 
(with the convention that $\til{u}^{(-1)}=0$) and $\til{u}^{(k)}$ is supported in $\til{\Lambda}_+$.  
By a direct application of \eqref{covarianceomegaik} and \eqref{changeofPhi}, we have $\gamma^*\til{u}^{(k)}=\til{u}^{(k)}$ for all $k$ and $\gamma\in \Gamma$, which implies that the distributions $\til{u}^{(k)}$ descend to distributions $u^{(k)}$ on $SM$ supported in $\Lambda_+$ and 
satisfying $(X+\la_0)u^{(k)}=u^{(k-1)}$ and $U_-u^{(k)}=0$. 
Finally, from the equation $(X+\la_0)u^{(k)}=u^{(k-1)}$ the wave-front set of
$u_{(k)}$ is contained in the annulator of $E_0=\rr X$  by elliptic regularity, 
and similarly from $U_-u^{(k)}=0$ it is also contained in the annulator of $U_-$, 
thus it has to be contained in $E_{u}^*$ (which is $E_+^*$ over $\Lambda_+$). 
Notice also that ${\pi_0}_*u^{(k)}=\varphi_k$ for each $k$.
This implies that $u^{(k)}\in V_0^k(\la_0)$ and the map ${\pi_0}_*:V_0^k(\la_0)\to {\rm Res}_{\Delta}^k(\la_0+1)$ is surjective.

\textbf{Case $\la_0\in -\demi-\nn$}. Note that $\Res^j_{\Delta}(1/2-k) = 0$
has been already shown in Proposition \ref{even}. In order to see $V_0^j(-1/2-k) =0$,
we use a similar argument as above: If $u\in V_0^1(\la_0)$ with $\la_0=-1/2-k$ for $k\in\nn$, then
$\varphi:={\pi_0}_*u$ solves $(\Delta_M-s_0(1-s_0))\varphi=0$ with $s_0=\la_0+1$. 
Furthermore,
by Lemma \ref{poissoniso}, there is $\omega\in \mc{D}(\mathbb{S}^1)$ supported in 
$\Lambda_\Gamma$ so that $\gamma^*\omega=N_\gamma^{\la_0} \omega$ for all $\gamma\in\Gamma$ and 
$\til{\varphi}:=\pi_{\Gamma}^*\varphi=\mc{P}_{\la_0}(\omega)$. 
Using \eqref{weakasymptotic0} and \eqref{-1/2-kbis}, we see that 
$\til{\varphi}=\rho_0^{1/2-k}F_1+\rho_0^{1/2+k}\log(\rho_0) F_2$ near each point 
$\til{p}\in\Omega_\Gamma$, where $F_j$ are smooth functions on $\hh^2\cup\Omega_\Gamma$ 
near $\til{p}$. This implies that 
$\varphi\in \rho^{1/2-k}C^\infty(\bbar{M})\oplus \rho^{1/2+k}\log(\rho)C^\infty(\bbar{M})$. 
But we also have $\varphi=\mc{E}_M(1/2+k)f$ if $f:=[\rho^{-1/2+k}\varphi]|_{\pl\bbar{M}}$ 
by the properties of $\mc{E}_M(s)$ (see \eqref{poissonM}).
As in the proof of claim 1), the fact that $\supp \omega \subset \Lambda_\Gamma$ implies
the vanishing of the $\rho^{-1/2+k}\log(\rho)$ terms in the asymptotic of $\varphi$. This implies by
\eqref{sppoints} that $\mc{S}_M(1/2+k)f=0$, which by \eqref{SMiso} shows that $f=0$, and 
thus $\varphi=0$, proving the claim 2).

\textbf{Case $\la_0=-1/2$.}  The arguments above show that each Ruelle 
resonant state $u\in V_0^{1}(-1/2)$ produces a quantum resonance $\varphi={\pi_0}_*u$ 
whose lift to $\hh^2$ is $\mc{P}_{\la_0}(\omega)\in \rho^{1/2}C^\infty(\bbar{M})$ 
at $s_0=1/2$ for some $\omega\in \mc{D}'(\mathbb{S}^1)$ supported in 
$\Lambda_\Gamma$ and satisfying $\gamma^*\omega=N_\gamma^{1/2}\omega$ for all 
$\gamma\in \Gamma$. If there is 
an element $u'\in V_0^{2}(\la_0)$, then like in Theorem \ref{noninteger} we have $\varphi':={\pi_0}_*u'$
satisfying $(\Delta_M-1/4)\varphi'=(2\lambda_0+1)\varphi =0$, with $\pi_{\Gamma}^*(\varphi')=:\mc{P}_{\la_0}(\omega')-\pl_{\la}\mc{P}_{\la_0}(\omega)$
for some $\omega'\in \mc{D}'(\mathbb{S}^1)$ supported in $\Lambda_\Gamma$. We have 
\[ \pl_{\la}\mc{P}_{-1/2}(\omega)(x)=\log(1-|x|^2)\mc{P}_{-1/2}(\omega)(x)+
\int_{\mathbb{S}^1}P(x,\nu)^{1/2}\log (|x-\nu|^2)\omega(\nu)d\nu\]
and, using that ${\rm supp}(\omega)\in \Lambda_\Gamma$ and \eqref{Sen1/2}, 
this implies directly that near a point $\til{p}\in \Omega_{\Gamma}$, 
\[
\pl_\la\mc{P}_{-1/2}(\omega)(x,\nu)=\sqrt{2}\log \rho_0(x)(\pl_{\la}\mc{S}(1/2)\omega)(\nu)+\mc{O}(1)
\]
as $|x|\to 1$. Since 
$\mc{P}_{\la_0}(\omega')\in \rho_0^{1/2}C^\infty(\hh^2\cup\Omega_{\Gamma})$, we get that 
$\varphi'=\log(\rho) f+\mc{O}(1)$ near $\pl\bbar{M}$, for some non-zero 
$f\in C^\infty(\pl\bbar{M})$. Therefore by 2) in Proposition \ref{caractreson} 
the function $\varphi'$ is not a generalized 
resonant state. 
This shows that $-1/2$ is a pole of order at most $1$ for 
$R_X(\la)$, there is no Jordan blocks
and ${\pi_0}_*: V_0^1(-1/2)\to {\rm Res}_{\Delta}^1(-1/2)$ is injective. The 
surjectivity works as for the cases above. 

\textbf{Case $\la_0\in -\nn_0$, injectivity.}
 By Proposition 
\ref{no_int_jord},  among generalized states at $\la_0=-n$ killed by $U_-$ and ${\pi_0}_*$, 
there can be only resonant states.
Let $u$ be a Ruelle resonant state satisfying $(X-n)u=0$, $U_-u=0$ and ${\rm supp}(u)\subset 
\Lambda_+$. We can always assume that $u$ is real valued: indeed, since the spectral 
parameter $-n$ is real valued, there is a basis of real-valued resonant states. 
The space $V_0(-n)$ can then be considered as a real vector space.

First, we assume that  ${\pi_0}_*u=0$. Consider the Fourier components $u_k$ in the 
fiber variables, then by the proof of Proposition \ref{no_int_jord}, we have 
$u_k=0$ for all $|k|<n$, and by \eqref{recursion}
we also get $\eta_-u_n=0$ and $\eta_+u_{-n}=0$. In particular $\bbar{\pl} u_n=0$ 
and $\pl u_{-n}=0$ when we view $u_{\pm n}$ as (distributional) sections of 
$\mc{K}^{\pm n}$. Notice by ellipticity that $u_{\pm n}$ are smooth and actually 
analytic. We will denote by $\til{u}_{\pm n}=\pi_{\Gamma}^*u_{\pm n}$ their 
lift to $\hh^2$, and in the ball model we can write $\til{u}_{n}=f_{n}dz^n$ and 
$\til{u}_{-n}=f_{-n}d\bar{z}^n$
for some holomorphic (resp. antiholomorphic) functions $f_n$ (resp. $f_{-n}$) on $\hh^2$ satisfying 
\[ \forall \gamma\in \Gamma, \,\, \gamma^*f_n=(\pl_z\gamma)^{-n}f_n \textrm{ and }
\gamma^*f_{-n}=(\bbar{\pl_{z}\gamma})^{-n}f_{-n}.\]
Take the distribution $v:=\Phi_-^{n}\til{u}$ where $\til{u}:=\pi_\Gamma^*u$: 
we get  $Xv=0$ and $U_-v=0$ thus there exists $\omega \in \mc{D}'(\mathbb{S}^1)$ 
so that $\mc{Q}_-\omega=v$ by \eqref{Qpmiso}. 
Since ${\rm supp}(\til{u})\subset B_-^{-1}(\Lambda_\Gamma)$, we get directly 
that $\supp(\omega)\subset \Lambda_\Gamma$ and
by \eqref{changeofPhi}, for any $\gamma\in \Gamma$, $\gamma^*\omega=
N_\gamma^{n}\omega$. We want to write the Fourier mode $\til{u}_k$ (in the fiber variable) 
in terms of $\omega$: Therefore we write $z=x_1+ix_2\in\hh^2$ and identify the unit tangent vector
$\demi (1-|z|^2)(\cos(\theta)\pl_{x_1}+\sin(\theta)\pl_{x_2})\in S_z\hh^2$ with $e^{i\theta}$. We get
\[ \begin{split}
\til{u}_k(z)= & (2\pi)^{-1}\Big(\int_{S_z\hh^2}\til{u}(z,e^{i\theta})e^{-ik\theta}d\theta\Big)
\Big(\frac{2}{1-|z|^2}\Big)^kdz^k\\
=& 2^k(2\pi)^{-1}\Big(\int_{0}^{2\pi} \omega(e^{i\alpha})e^{-ik\alpha}
|1-ze^{-i\alpha}|^{2(n-1)}\frac{(\bar{z}e^{i\alpha}-1)^k}{(1-ze^{-i\alpha})^k} d\alpha  \Big)
\frac{dz^k}{(1-|z|^2)^{k+n-1}}
\end{split}
\]  
where we used the change of variable $e^{i\theta}=B_z^{-1}(e^{i\alpha})$ defined by \eqref{formulaBinv}.
For $|k|<n$ we have $\til{u}_k=0$, thus by evaluating at $z=0$ we get
\begin{equation}\label{modek0}
\forall k \in (-n,n),\quad  0=\int_{0}^{2\pi} \omega(e^{i\alpha})e^{-ik\alpha}d\alpha.
\end{equation}
For $k=n$, we know that $\til{u}_n=f_ndz^n$ with $f_n$ holomorphic on $\hh^2$, thus 
we deduce that 
\begin{equation}\label{fnz} 
f_n(z)=\frac{(-2)^{n}}{2\pi}\int_{0}^{2\pi} \omega(e^{i\alpha})
\frac{e^{-in\alpha}}{1-ze^{-i\alpha}} d\alpha.
\end{equation}
We deduce from this that $f_n(z)$ has a series expansion converging in $|z|<1$ given by 
\[ 
f_n(z)=\frac{(-2)^{n}}{2\pi}\sum_{k=0}^\infty \omega_kz^k , \quad \omega_k:= \cjg \omega,e^{-i(n+k)\alpha}\cjd,
\]
where here we notice that $|\omega_k|=\mc{O}((1+|k|)^N)$ for some $N$, since $\omega$
is in some Sobolev space on $\mathbb{S}^1$ (here and below, $\cjg\cdot,\cdot\cjd$ denotes 
the bilinear distributional pairing on $\mathbb{S}^1$ with respect to the natural measure 
$d\alpha$ of mass $2\pi$). From \eqref{fnz} and since ${\rm supp}(\omega)\subset \Lambda_{\Gamma}$, 
we see that $f_n(z)$ extend holomorphically to $\cc\setminus \Lambda_\Gamma$. The section 
$f_n(z)dz^n$ is holomorphic on $\cc\setminus {\Lambda_\Gamma}$ and $\Gamma$ equivariant, 
thus descend to a holomorphic section of $\mc{K}^n$ on $M_2$ with the notation of Section 
\ref{cccsurface}. Similarly, we get (using that $\omega$ is real-valued)
\[
f_{-n}(z)=\frac{(-2)^{n}}{2\pi}\int_{0}^{2\pi} \omega(e^{i\alpha})
\frac{e^{in\alpha}}{1-\bar{z}e^{i\alpha}} d\alpha= \frac{(-2)^{n}}{2\pi}\sum_{k=0}^\infty \bbar{\omega_k}\bar{z}^k=\bbar{f_n(z)}.
\]
Now for each $\psi\in C^\infty(\mathbb{S}^1)$, write $\psi=\sum_{k\in\zz}\psi_ke^{ik\theta}$, 
then for 
$r<1$ we get
\[
\int_{0}^{2\pi}f_n(re^{i\theta})\psi(\theta)d\theta=(-2)^n
\sum_{k\geq 0}\omega_kr^k\psi_{-k}
\]
and this converges to $(-2)^n\cjg \Pi_+(\omega e^{-in\theta}),\psi\cjd $ as $r\to 1$ where 
$\Pi_+$ is the Szeg\"o projector, i.e. the projector $\indic_{[0,\infty)}(-i\pl_{\theta})$ 
on the non-negative Fourier modes on $\mathbb{S}^1$. By using \eqref{modek0},  this means 
that $f_n$ has a weak limit on $\mathbb{S}^1$ and 
$f_n|_{\mathbb{S}^1}=(-2)^ne^{-in\theta}\Pi_+(\omega)$ in the weak sense. 
Similarly, we have $f_{-n}|_{\mathbb{S}^1}=(-2)^ne^{in\theta}\Pi_-(\omega)$ if
$\Pi_-:=\indic_{(-\infty,0]}(-i\pl_{\theta})$.
Now by injectivity of the Poisson-Helgason transform at the spectral parameter $0$ we obtain 
$C_n\in \rr^*$ such that 
\[
\mc{P}_{0}(\omega)=\mc{P}_{0}(\Pi_+(\omega))+\mc{P}_{0}(\Pi_-(\omega))=
C_n(z^nf_n+\bbar{z^nf_n})
\]
But since $\omega|_{\Omega_{\Gamma}}=0$ the harmonic function $z^nf_n+\bbar{z^nf_n}$
vanishes on $\Omega_{\Gamma}$ thus $e^{in\theta}f_n|_{\Omega_{\Gamma}}\in i\rr$, which means 
that $i^{n+1}\iota_{\pl \bbar{M}}^*u_n$ is a real-valued symmetric tensor on $\pl \bbar{M}$. 
This is equivalent to say that $i^{n+1}u_n\in H_n(M)$. 
Moreover, the map $u\mapsto i^{n+1}u_n\in H_n(M)$ is injective since $u_n=0$ implies $\omega=0$ and thus 
$u=0$. 

\textbf{Case $\la_0\in-\nn$, surjectivity.} Conversely, let $u_n\in i^{-n-1}H_n(M)$ and consider its lift $\til{u}_n=f_ndz^n$ to $\hh^2$. The 
holomorphic function $f_n$ satisfies $f_n(\gamma(z))=(\pl_z\gamma(z))^{-n}f_n(z)$ 
for all $\gamma\in\Gamma$, or equivalently $\gamma^*\til{u}_n=\til{u}_n$, and we can assume 
that $e^{in\theta}f_n|_{\Omega_\Gamma}\in i\rr$. 
The tensor $u_n$ is bounded on $\bbar{M}$ with respect to the hyperbolic metric thus 
$|f_n(z)dz^n|_{g_{\hh^2}}\in L^\infty(\hh^2)$, and since $|2dz^n/(1-|z|^2)^n|_{g_{\hh^2}}=1$, 
we deduce that $z\mapsto f_n(z)(1-|z|^2)^n$ is bounded in the unit disk and therefore $f_n$ is 
tempered. In particular there exists $\omega_\pm\in\mc{D}'(\mathbb{S}^1)$ 
 so that $\mc{P}_{0}(\omega_+)=z^nf_n$ and $\mc{P}_{0}(\omega_-)=\bbar{z^nf_n}$, and in fact $\omega_-=\bbar{\omega_+}$. We have $\omega:=\omega_++\omega_-$ which is supported on $\Lambda_\Gamma$ since $\omega$ is the boundary value of ${\rm Re}(z^nf_n)$ to $\mathbb{S}^1$ in the weak sense.
Next we want to describe the covariance of $\omega$ with respect to each $\gamma\in \Gamma$: write 
$\gamma(e^{i\alpha})=e^{i\mu(\alpha)}$ for the action on $\mathbb{S}^1$, then $|d\gamma(e^{i\alpha})|=\mu'(\alpha)$ 
we have  $\gamma^*(z^nf_n(z))=(z\frac{\pl_z\gamma(z)}{\gamma(z)})^{-n}f_n(z)$, which when restricted on $\mathbb{S}^1$ gives 
\begin{equation}\label{equivome} 
\gamma^*\omega =\Big(\frac{-i\pl_\alpha (\gamma(e^{i\alpha})}{\gamma(e^{i\alpha})}\Big)^{-n}\omega=
|d\gamma|^{-n}\omega.\end{equation}
Thus $\til{u}=\Phi_-^{-n}\mc{Q}_-(\omega)\in \mc{D}'(\hh^2)$ solves $(X+n)\til{u}=0$ and $U_-\til{u}=0$, 
it is $\Gamma$-invariant by using \eqref{equivome}  and has support in $\til{\Lambda}_+=B_-^{-1}(\Lambda_\Gamma)$. This implies that 
$\til{u}$ descends to a Ruelle resonance $u$, with support in $\Lambda_-$ and with ${\rm WF}(u)\subset E_+^*$ by ellipticity arguments as before. The map $u_n\mapsto u$ is injective since 
$u_n\to \omega$ is injective.
Moreover by construction $\omega$ has vanishing $k$-Fourier components for all $|k|<n$ on $\mathbb{S}^1$, thus it is in the kernel of $\mc{P}_{-n}$, which means that 
${\pi_0}_*u=0$, and thus ${\pi_k}_*u=0$ for all $|k|<n$ by Lemma \ref{recurrence_relation} (just like in the proof of Proposition \ref{no_int_jord}).

To conclude the proof we have to prove that a Ruelle resonant state $u\in V_0^1(-n)$ 
satisfying $u_0={\pi_0}_*u\not=0$ does not exist. If $u_0\not=0$, we 
have by Lemma \ref{poissoniso} that $(\Delta_M-n(1-n))u_0=0$.
The lift $\til{u}_0=\pi_{\Gamma}^*u_0$ to $\hh^2$
must be in ${\rm Ran}(\mc{P}_{-n})$, which by \eqref{kerP-n} implies that 
$\til{u}_0\in \rho_0^{1-n}C^\infty(\bbar{\hh}^2)$ is an element of $\ker(\Delta_{\hh^2}-n(1-n))$
of the form
\[ \til{u}_0(z)=(1-|z|^2)^{1-n}L_n(z,\bar{z})\]
for some polynomial $L_n$ of degree $2n-2$. Since for each $\gamma\in\Gamma$ we have
$(1-|\gamma(z)|^2)=(1-|z|^2)|\gamma'(z)|$, we deduce from $\gamma^*\til{u}=\til{u}$ that 
\[
L_n(\gamma(z),\bbar{\gamma(z)})=|\gamma'(z)|^{n-1}L_n(z,\bar{z}).
\]
Taking $z=z_\pm$ to be the two fixed points of $\gamma$, we deduce that 
$L_n(z_\pm,\bbar{z_\pm})=0$ since $|\gamma'(z_\pm)|\not=1$. Therefore $L_n$ is a polynomial 
in $(z,\bar{z})$ vanishing on the limit set $\Lambda_\Gamma$, and thus it vanishes on the whole $\mathbb{S}^1$ by analyticity, if $\Gamma$ is non-elementary. We deduce that $\til{u}_0=\mc{O}(\rho_0^{2-n})$ at 
$\mathbb{S}=\pl\hh^2$, and thus $u\in\rho^{2-n} C^\infty(\bbar{M})$. From the form of $\Delta_M$ near $\rho=0$ given by \eqref{DeltaM}, a Taylor expansion in $\rho=0$ of the equation $(\Delta_M-n(1-n))u_0=0$ implies 
that actually $u_0\in \rho^{n}C^\infty(\bbar{M})\subset L^2(M)$, and therefore $u_0=0$ since 
$n(1-n)\leq 0$, leading to a contradiction. 
\end{proof}

\textbf{The case of an elementary group.} Let us  briefly discuss the case of an elementary group generated by one transformation 
$\gamma\in {\rm PSL}_2(\rr)$ where we only need to focus on the spectral points $-n\in-\nn$. We can assume that the two fixed points of $\gamma$ in the disk model are $\pm 1$ (with $-1$ being the repulsive one). Mapping
the disk to the upper half-plane so that $\mp 1$ is mapped to $0$ and $\pm 1$ to $\infty$, the transformation $\gamma$ is conjugated to $\gamma_\mp: z\mapsto e^{\pm \ell}z$ for some $\ell>0$.
By the discussion in the proof of Theorem \ref{classicquantic}, the Ruelle resonant states at $-n\in-\nn$ are in correspondence with the $\omega\in \mc{D}'(\mathbb{S}^1)$ supported in $\{-1,1\}$ and satisfying $\gamma^*\omega=|N_\gamma|^{n}\omega$. There can only be Dirac masses and its derivatives at $\pm 1$. To analyse the restriction of $\omega$ on a small neighborhood $O_{\mp}\subset \mathbb{S}^1$ of $\mp 1$, we use the upper half-plane model, so $\omega|_{O_{\mp}}$ becomes (after conjugation) a distribution 
$\omega_\mp\in \mc{D}'(\rr)$ supported in $\{0\}$ and satisfying $\gamma_\mp^*\omega_\mp=e^{\mp n\ell}\omega_\mp$. The only solutions are of the form $c_\mp \delta_0^{(n-1)}$ for $c_\mp \in\cc$ if $\delta^{(j)}_0$ denotes the $j$-th derivative of the Dirac mass $\delta_0$ at $0$ in $\rr$. 
The same argument shows that there is no generalised resonant state (no Jordan blocks) for the operator $X$, i.e. $V_0^j(-n)=V_0^1(-n)$ for $j>1$, and the space of Ruelle resonant states at $-n$ is exactly of 
dimension $2$. The Poisson kernel in the half space model is 
$P(x,y,x')=\frac{y}{y^2+|x-x'|^2}$, where $z=x+iy\in\hh^2$ and $x'\in\rr$ is the point at infinity. We can just compute that in the upper half-space model, where the repulsive fixed point $-1\in \mathbb{S}^1$ is mapped to $0$ and the repulsive $+1\in \mathbb{S}^1$ to $\infty$, we have 
\[
\mathcal P_{-n}(\omega_-)=c_-\cjg  \delta_{x'=0}^{(n-1)} , \tfrac{(y^2+|x-x'|^2)^{n-1}}{y^{n-1}}\cjd=c_-Q(x,y)/y^{n-1}.
\]
Here $Q$ is a homogeneous polynomial of degree $n-1$ in $\cc$, satisfying $Q(-x,y)=(-1)^{n-1}Q(x,y)$. The same argument shows that in the half-plane model where $-1\in \mathbb{S}^1$ is mapped to $\infty$ and $+1\in \mathbb{S}^1$ to $0$, shows that 
$\mathcal P_{-n}(\omega_+)=c_+Q(x,y)/y^{n-1}$. On the other hand, the mapping from one half-plane model to the other can be chosen to be $A: z\mapsto -1/z$. Due to the parity property of $Q(x,y)$ in $x$, we easily see that $Q(A(z))/({\rm Im}(Az))^{n-1}=(-1)^{n-1}Q(z)/({\rm Im}(z))^{n-1}$. This shows that ${\pi_0}_*$ maps $V_0^1(-n)$ to a $1$-dimensional space ${\rm Res}_{\Delta}^1(-n+1)$ of quantum resonant states. The elements in $\ker {\pi_0}_*\cap V_0^1(-n)$ are in correspondence with $H_n(M)$, as we discussed in the proof of Theorem \ref{classicquantic}, and this space has real dimension equal to $1$. The quantum resonances at $s=-n+1$ are computed in \cite[Chapters 5.1 and 8.2]{Bo}: there are poles of order $2$, there is a $1$-dimensional space of quantum resonant states and the multiplicity is equal to $2$ due to the Jordan block. The correspondence between classical and quantum resonant states is thus different from the non-elementary group case for these particular points.\\

Finally exactly the same proof as Corollary \ref{otherbands} gives the full Ruelle resonance spectrum.
\begin{corr}\label{otherbands2}
Let $M=\Gamma\backslash \hh^{2}$ be a smooth oriented convex co-compact hyperbolic surface 
and let $SM$ be its unit tangent bundle. Then for each $\la_0\in \cc$ with ${\rm Re}(\la_0)\leq 0$, 
$k\in\nn_0$, and $j\in\nn$, the operator $U_+^k$ is injective on $V_0^j(\la_0+k)$ and we get 
\[ V_{k}^j(\la_0)=\bigoplus_{\ell=0}^k U_+^\ell(V_0^j(\la_0+\ell)).\]
\end{corr}

\section{Zeta functions}\label{sec:zeta_functions}

The zeta function of the flow is defined by 
\begin{equation}\label{Z0} 
Z_X(\la)=\exp\Big(-\sum_{\gamma_0}\sum_{k=1}^\infty \frac{1}{k}\frac{e^{-\la k \ell(\gamma_0)}}{|\det(1-P(\gamma_0)^k)|}\Big)
\end{equation}
where $\gamma_0$ are primitive closed geodesics and $P(\gamma_0)$ is the linearized Poincar\'e map
of the geodesic flow on this geodesic. The function converges for ${\rm Re}(\la)>\delta_\Gamma$ where 
$\delta_{\Gamma}<1$ is the Hausdorff dimension of the limit set $\Lambda_{\Gamma}$ (see \cite{Pa}).
By \cite[Theorem 4]{DyGu}, the function $Z_X(\la)$ admits a holomorphic extension to $\la\in\cc$ with 
zeros at Ruelle resonances and the order of a Ruelle resonance $\la_0$ as a zero of $Z_X(\la)$ is given by 
\begin{equation}\label{ordre} 
{\rm ord}_{\la_0}(Z_X(\la))={\rm Rank}(\Pi^X_{\la_0})
\end{equation}
where $\Pi^X_{\la_0}=-{\rm Res}_{\la_0}R_X(\la)$ is the projector on generalized Ruelle resonant states.
In particular we deduce from \eqref{ordre} and Theorem \ref{otherbands2} the 
\begin{prop}\label{zerosZ_X}
The order of $\la_0$ as a zero of $Z_X(\la)$ is given by 
\[{\rm ord}_{\la_0}(Z_X(\la))=\dim {\rm Res}_{X}(\la_0)=\sum_{p\in \nn_0} 
\dim ({\rm Res}_{X}(\la_0+p)\cap \ker U_-).\]
where ${\rm Res}_X(\la_0)=\cup_{j\geq 1}{\rm Res}_X^j(\la_0)$ is the space of generalized resonant 
states at $\la_0$, with ${\rm Res}_X^j(\la_0)$ defined in \eqref{Resk2}.
\end{prop}
The Selberg zeta function $Z_{S}(\la)$ is defined by 
\begin{equation}\label{Selberg} 
Z_{S}(\la):= \exp\Big( -  \sum_{\gamma_0}\sum_{k=1}^\infty \frac{1}{k}\frac{e^{-\la k\ell(\gamma_0)}}{\det (1-P_s(\gamma_0)^k)}
\Big)
\end{equation}
where the sum is over all primitive closed geodesics and $P_s(\gamma_0)=P(\gamma_0)|_{E_s}$ 
is the contracting part of $P(\gamma_0)$. For each closed geodesic $\gamma$ on $M=\Gamma\backslash\hh^2$, 
there is an associated conjugacy class in $\Gamma$, with a representative that we still denote by 
$\gamma\in \Gamma$ and whose axis in $\hh^{2}$ descends to the geodesic $\gamma$; the linear 
Poincar\'e map along this closed geodesic is easy to compute since $\gamma$ is conjugated to 
$z\mapsto e^{\ell(\gamma)}z$ in the upper half-space model of $\hh^2$. Using this expression 
we get $P(\gamma_0)|_{E_s}=e^{-\ell(\gamma_0)}{\rm Id}$ and $P(\gamma_0)|_{E_u}=e^{\ell(\gamma_0)}{\rm Id}$
thus 
\begin{equation}\label{poincare} 
\begin{split}
 e^{-\frac{1}{2}k\ell(\gamma_0)}\det (1-P_s(\gamma_0)^k)^{-1}
= e^{-\frac{1}{2}k\ell(\gamma_0)}(1-e^{-k\ell(\gamma_0)})^{-1}=\sum_{p=0}^\infty e^{-k\ell(\gamma_0)(1/2+p)}
\end{split}\end{equation}
and $|\det(1-P(\gamma_0)^k)|=e^{k\ell(\gamma_0)}\det(1-P_s(\gamma_0)^k)^2$. This implies the formula 
\begin{equation}\label{relationzetas}
Z_X(\la)=\prod_{p=1}^\infty Z_S(\la+p), \quad Z_S(\la)=\frac{Z_X(\la-1)}{Z_X(\la)}.
\end{equation} 
By combining \eqref{relationzetas}  with Proposition \ref{zerosZ_X} and 
Theorem \ref{classicquantic}, we obtain 
\begin{corr}\label{selberg}
Let $M=\Gamma\backslash \hh^2$ be a convex co-compact oriented hyperbolic smooth 
surface and assume $\Gamma$ is non-elementary. Then its Selberg zeta function $Z_S(s)$ 
is holomorphic with zeros given by:\\
1) quantum resonances $s_0 \notin -\nn_0$ with order 
\[
{\rm ord}_{s_0}(Z_S(s))={\rm Rank}({\rm Res}_{s_0}R_{\Delta}(s))
\]
2) negative integers $-n\in-\nn_0$ with order 
\[ {\rm ord}_{-n}(Z_S(s))=\dim_{\rr} H_{n+1}(M)=\left\{
\begin{array}{ll}
(2n+1)|\chi(M)| & \textrm{ if }n>0,\\
|\chi(M)|+1 & \textrm{ if }n=0
\end{array}\right.\]
where $H_{n+1}(M)$ are defined by \eqref{HnH-n}.
\end{corr}
The description of the zeros of $Z_\Gamma(s)$ in this setting was also done by Borthwick-Judge-Perry \cite{BJP} (including the case with cusps). 
We remark that in \cite{BJP} the topological contribution in the order of $Z_S(s)$ at $s=0$ is $|\chi(M)|$. 
Furthermore there is a spectral contribution coming from the multiplicity of $0$ as a quantum resonance.
This multiplicity is exactly $1$ by the proof of Theorem 1.2 in \cite{GuGu}, the resonant states being the constants  (see the discussion after the proof of \cite[Theorem 1.2]{GuGu}), which matches with part 2) in Corollary \ref{selberg}.
For the points $s=-n$ with $n\in\nn$, \cite{BJP} obtain a zero of order $(2n+1)|\chi(M)|+\nu(-n)$ where 
$\nu(-n)={\rm Rank}({\rm Res}_{-n}R_{\Delta}(s))$ is the order of $-n$ as a resonance, and thus 2) of 
Corollary \ref{selberg} implies that $\nu(-n)=0$ for non-elementary groups $\Gamma$.

\end{document}